\documentclass[a4, reqno]{amsart}
\usepackage{fullpage}

\usepackage{amsmath}
\usepackage{amsthm}
\usepackage{amsfonts}
\usepackage{amssymb}
\usepackage{tikz-cd}
\usepackage[utf8x]{inputenc}
\usepackage{enumerate}
\usepackage[english]{babel}
\usepackage{esint}
\usepackage{hyperref}
\usepackage{bm}
\usepackage{bbm}
\usepackage[mathscr]{eucal}

\theoremstyle{plain}
\newtheorem{theorem}{Theorem}[section]
\newtheorem{lemma}[theorem]{Lemma}

\newtheorem{proposition}[theorem]{Proposition}
\theoremstyle{definition}
\newtheorem{definition}[theorem]{Definition}
\theoremstyle{remark}

\newtheorem{remark}[theorem]{Remark}

\numberwithin{equation}{section} 

\DeclareMathOperator*{\esssup}{ess\:sup}

\DeclareMathOperator{\trace}{Tr}

\newcommand{\vr}{\varrho}
\newcommand{\vu}{\textbf{\textup{u}}}
\newcommand{\vv}{\textbf{\textup{v}}}
\newcommand{\vt}{\vartheta}
\newcommand{\vm}{\textbf{m}}
\newcommand{\vq}{\textbf{q}}
\newcommand{\ve}{\varepsilon}
\newcommand{\vS}{\mathbb{S}}
\newcommand{\vU}{\textbf{\textup{U}}}

\newcommand{\vphi}{\varphi}
\newcommand{\bfphi}{\bm{\varphi}}
\newcommand{\vrex}{\widetilde{\varrho}_{\varepsilon}}
\newcommand{\vuex}{\widetilde{\textbf{\textup{u}}}_{\varepsilon}}
\newcommand{\vtex}{\widetilde{\vartheta}_{\varepsilon}}

\newcommand{\Div}{{\rm div}_x}
\newcommand{\Grad}{\nabla_x}
\newcommand{\DerTime}{\partial_t}
\newcommand{\dx}{\textup{d}x}
\newcommand{\dt}{\textup{d}t}

\def\softd{{\leavevmode\setbox1=\hbox{d}%
		\hbox to 1.05\wd1{d\kern-0.4ex{\char039}\hss}}}

\allowdisplaybreaks[1]

\begin{document}
	\title{Low Mach number limit on perforated domains\\for the evolutionary Navier--Stokes--Fourier system}
	\author{Danica Basari\'{c} $^{\ast}$  \ and \ Nilasis Chaudhuri $^{\dagger}$}
	\date{}
	
	\address{$^{\ast}$ Dipartimento di Matematica, Politecnico di Milano \\
		Via E. Bonardi 9, 20133 Milano, Italy}
	\email{danica.basaric@polimi.it}

	\address{$^{\dagger}$ Faculty of Mathematics, Informatics and Mechanics, University of Warsaw, ul. Banacha 2, Warsaw 02-097, Poland}
	\email{nchaudhuri@mimuw.edu.pl}
	
	\thanks{$^{\ast}$ The work of D. B. was supported by the Czech Sciences Foundation (GA\v CR), Grant Agreement 21--02411S. The Institute of Mathematics of the Academy of Sciences of the Czech Republic is supported by RVO:67985840. The Department of Mathematics of Politecnico di Milano is supported by MUR “Excellence Department 2023-2027”}
	\thanks{$^{\dagger}$ The research of NC was supported by the EPSRC Early Career Fellowship  EP/V000586/1, and by the ``Excellence Initiative Research University (IDUB)" program at the University of Warsaw. }	
	\thanks{The authors wish to thank Prof. Eduard Feireisl for the helpful advice and discussions.}
	
	\maketitle 
	
	\begin{abstract}
		We consider the Navier--Stokes--Fourier system describing the motion of a compressible, viscous and heat-conducting fluid on a domain perforated by tiny holes. First, we identify a class of dissipative solutions to the Oberbeck-Boussinesq approximation as a low Mach number limit of the primitive system. Secondly, by proving the weak--strong uniqueness principle, we obtain strong convergence to the target system on the lifespan of the strong solution. 
	\end{abstract}

	\section{Introduction}
	
	The aim of this work is to study the asymptotic analysis of the scaled Navier--Stokes--Fourier system  in a domain perforated with tiny holes. More precisely, we consider the physical situation corresponding to the \textit{low stratification} of a fluid, i.e. the equations describing the motion of a compressible viscous fluid are scaled by a small Mach number $\textup{Ma}=\ve^m$ and Froude number $\textup{Fr}= \ve^{m/2}$ for a fixed positive $m$; in addition, we suppose that the fluid is confined to a bounded spatial domain perforated by many holes, each of them properly contained in a ball of radius $\ve^{\alpha}$ and having mutual distance $\ve$ for some $\alpha>1$. We keep other characteristic numbers Strouhal number, Reynolds number and P\'eclet number as unity. Our goal consists in analyzing what happens when we let $\ve$ go to zero. \par 
	In the absence of holes, the problem reduces to a classical asymptotic analysis problem in a fixed domain, mainly the low Mach number limit, which is also referred to as the incompressible limit in the context of compressible systems in the literature. The first approach, proposed by Klainarman and Majda \cite{KlaMaj}, is based on classical or strong solutions of the compressible system and proves that the limit is an incompressible system. This approach has been followed by Alazard \cite{Ala} to analyze the low Mach number limit for the Navier--Stokes--Fourier system. On the other hand, based on global-in-time weak solutions, Lions and Masmoudi \cite{LioMas}, and Desjardins et al. \cite{DeGrLiMa} studied the low Mach number limit for the compressible Navier--Stokes system and they obtained the incompressible Navier--Stokes system as a limit. This approach has also been extended to the Navier--Stokes--Fourier system. We refer to the monograph of Feireisl and Novotn\'y \cite{FeiNov}, where different multiscale problems (like, $\textup{Ma} = \textup{Fr}$ and $\sqrt{\textup{Ma}} = \textup{Fr}$) are addressed. These multiple scalings explain the stratification of fluid. \par  
	On the other hand, for a fixed Mach and Froude number, the problem coincides with the homogenization problem for fluid dynamics, which aims to describe the macroscopic behavior of microscopically heterogeneous systems. In general, the limiting behavior depends on the size and mutual distance of holes, that is, the relation between the radius of holes $ \ve^\alpha $ and mutual distance $ \ve$. For incompressible stationary Stokes and Navier--Stokes problems with periodically distributed holes, in his seminal works Allaire (in \cite{All}, \cite{All2}, see also Tartar \cite{Tar}) proved that in the case of ``large'' holes, that is, $1\leq \alpha <3$, the limit system is governed by the Darcy law, while for ``tiny'' holes, that is, $\alpha >3$, the limit system remains the same as the original one. The critical case $\alpha=3$ leads to Brinkmann's law. Similar results hold in the context of evolutionary Stokes and incompressible Navier--Stokes systems, as shown by Mikeli\'c \cite{Mik} and Feireisl, Namlyeyeva, and Ne\v casov\'a \cite{FeiNamNec}. All of the above results are in three dimensions.\par 
	In the case of compressible fluids, the situation is more complex than its incompressible counterpart. For the barotropic Navier--Stokes system with Strouhal number proportional to $\ve^2$ and the diameter of the holes proportional to their mutual distance (i.e., ``large'' holes with $\alpha=1$), the problem was considered by Masmoudi \cite{Mas} who deduced that the limit system is the porous medium equation with the nonlinear Darcy's law. For the heat-conducting fluid (Navier--Stokes--Fourier) system with the same $\alpha$, Feireisl, Novotn\'y, and Takahashi \cite{FeiNovTak} achieved similar results. Recently, the case of tiny holes ($\alpha>3$) has been studied in several papers, and the limit problem was identified as the same as in the perforated domain in three dimensions. Along with the mutual distance and diameters of the holes, the results also depend on the adiabatic exponent $\gamma$. For the steady compressible Navier--Stokes equations, Feireisl and Lu \cite{FeiLu} considered $\gamma>3$, while Diening, Feireisl, and Lu \cite{DieFeiLu} considered $\gamma>2$. Lu and Schwarzacher \cite{LuSch} studied the evolutionary compressible Navier--Stokes equations and proved that the presence of tiny holes is negligible for $\gamma>6$, which was recently improved to $\gamma>3$ by Oschmann and Pokorn\'y \cite{OscPok}.  Lu and Pokorn\'{y} \cite{LuPok} proved that the size of holes is negligible in the context of the stationary Navier--Stokes--Fourier system, while the same result was achieved by Pokorn\'{y} and Sk\v{r}\'{i}\v{s}ovsk\'{y} \cite{PokSkr} for the evolutionary case, considering a pressure of the type $p(\vr, \vt)= \vr^{\gamma}+ \vr \vt + \vt^4$ with $\alpha>7$, and $\gamma>6$. Recently, Oschmann and Pokorn\'y \cite{OscPok} improved the above results for the evolutionary compressible Navier--Stokes and Navier--Stokes--Fourier systems to $\alpha>3$, and $\gamma>3$. For the Navier--Stokes system, the challenging situation with dimension two was considered by Ne\v casov\' a and Pan \cite{NecPan} for $\gamma>2$ and by Ne\v casov\' a and Oschmann \cite{NecOsc} for $\gamma>1$. Recently, Bella and Oschmann \cite{BelOsc} considered the case of randomly perforated domains with the random size of holes.\par
	For the low Mach number limit of the compressible Navier--Stokes equation in a perforated domain, H\"ofer, Kowalczyk and Schwarzacher \cite{HofKowSch} recover Darcy's law as a limit of the system by considering $4m>3(\gamma+2)(\alpha-1)$, where the adiabatic exponent $\gamma\geq 2$. Very recently, Bella, Feireisl, and Oschmann \cite{BelFeiOsc} proved that in the case of tiny holes ($ \alpha > 3$) and under the hypothesis $\frac{2m}{\gamma}>\alpha$ with the adiabatic exponent $\gamma> \frac{3}{2}$, weak solutions of the compressible Navier--Stokes equation converge to a dissipative solution of the incompressible Navier--Stokes system for well-prepared initial data. Eventually, the use of the weak-strong uniqueness property ensures the convergence of weak solutions of the primitive system towards the strong solution for the target system, at least in the interval of existence of the strong solution.\par 
	To the best of the authors' knowledge, this is the first time that the low Mach number limit and the homogenization of the spatial domain have been performed simultaneously for the Navier--Stokes--Fourier system, enabling the consideration of general forms for pressure. Following the idea proposed in previous work \cite{BelFeiOsc}, we consider the weak solution for the Navier--Stokes--Fourier system and take the limit as $\varepsilon \rightarrow 0$ to obtain a dissipative solution of the Oberbeck-Boussinesq system for well-prepared initial data. Subsequently, we apply the weak-strong uniqueness property to ensure convergence to the strong solution of the target system, at least in the interval of existence of the latter. The two main ingredients we use are based on the restriction operator constructed by Diening et al. \cite{DieFeiLu} and a suitable extension operator for state variables, mainly for temperature, as suggested by Lu and Pokorn\'y \cite{LuPok} in Sobolev spaces, and later extended by Pokorny and Sk\v r\'i\v sovsk\'y\cite{PokSkr} in time dependent Sobolev spaces. \par 
	It is worth mentioning that the analysis presented in this article is largely motivated by the work of \cite{FeiNov}, where the author successfully recovered the same system without any presence of holes. Additionally, they were able to get the weak solution of the target system by employing a suitable analysis of the acoustic equation to establish the convergence of the convective terms $\vr_{\ve} \vu_\ve \otimes \vu_\ve$ towards $\overline{\vr} \vu \otimes \vu$, with $\overline{ \vr}$ and $\vu$ being the limits of the densities $\vr_{ \ve}$ and velocities $\vu_{\ve}$, respectively.
	Meanwhile, our approach takes a slightly different path to navigate the challenges associated with the convective term. We achieve this by considering only the weak limit of $\vr_{\ve} \vu_\ve \otimes \vu_\ve$ in the weak formulation of the target system, ultimately resulting in obtaining solely a dissipative solution for the target system. Nevertheless, the presence of the holes makes the whole analysis more challenging and the same procedure developed in \cite{FeiNov} is hardly applicable in this context; for a slightly different setup, such limitation on the convergence of nonlinear term $\vr_\ve \vu_\ve \otimes \vu_\ve$ is evident from the work of Masmoudi \cite{Mas}, and in particular, from the non-trivial difficulties arising from the use of the restriction operator $\mathcal{R}_{\ve}$, defined in Proposition \ref{main result} below. Therefore, the convergence of the nonlinear term $\vu_\ve \otimes \vu_\ve$ to $\vu \otimes \vu$ may even fail after the lifetime of the strong solution, making the dissipative solution the best option one can hope to achieve globally in time. Finally, it is worth pointing out that in \cite[Chapter 5]{FeiNov}, the authors were able to recover the result even for ill-prepared initial data; however, weak solutions arising from these type of data may not satisfy the energy inequality, which plays a significant role in proving the weak-strong uniqueness of the Oberbeck-Boussinesq system. Thus, following \cite[Section 5.5.4]{FeiNov}, we restrict ourselves to the framework of \textit{well-prepared} initial data; see Section \ref{Well-prepared data} for more details.
	\par 
	
	\subsection{Primitive system}
	
	Let us consider the scaled Navier--Stokes--Fourier system with small Mach number $\textup{Ma}=\ve^m$ and Froude number $\textup{Fr}= \sqrt{\textup{Ma}}= \ve^{m/2}$, with the positive real number $m$ fixed; specifically, we will consider 
	\begin{align}
		\DerTime \vr + \Div (\vr \vu) &=0, \label{continuity equation}\\
		\DerTime (\vr \vu) + \Div (\vr \vu \otimes \vu) + \frac{1}{\ve^{2m}}\Grad p(\vr, \vt) &= \Div \vS (\vt, 	\Grad \vu) + \frac{1}{\ve^{m}} \vr \Grad G, \label{balance of momentum}\\
		\DerTime \big(\vr e(\vr, \vt)\big) + \Div \big(\vr e(\vr, \vt) \vu\big) + \Div \vq(\vt, \Grad \vt)&= \ve^{2m}\  	\vS (\vt, \Grad  \vu) : \Grad \vu - p(\vr, \vt) \Div \vu. \label{balance internal energy}
	\end{align}
	Here the unknown variables are the density $\vr=\vr(t,x)$, the velocity $\vu=\vu(t,x)$ and the absolute temperature $\vt= \vt(t,x)$ of the fluid, while the pressure $p=p(\vr, \vt)$ and the internal energy $e=e(\vr, \vt)$ are related to a third quantity, the entropy $s=s(\vr, \vt)$, through Gibb's relation
	\begin{equation} \label{Gibb's relation}
		\vt Ds = De + p D \left( \frac{1}{\vr} \right).
	\end{equation}
	Due to the aforementioned relation, equation \eqref{balance internal energy} can be  equivalently rewritten as
	\begin{equation*}
		\begin{aligned}
			\DerTime \big(\vr s(\vr, \vt)\big) &+ \Div \big(\vr s(\vr, \vt) \vu\big) + \Div \left(\frac{\vq(\vt, \Grad \vt)}{\vt}\right)\\
			&= \frac{1}{\vt}\left(\ve^{2m} \ \vS (\vt, \Grad  \vu) : \Grad \vu - \frac{\vq(\vt, \Grad \vt) \cdot \Grad \vt}{\vt}\right). 
		\end{aligned}
	\end{equation*}
	We suppose that the fluid is Newtonian, meaning that the viscous stress tensor $\vS= \vS (\vt, \Grad  \vu)$ is given by 
	\begin{equation} \label{viscous stress tensor}
		\vS(\vt, \Grad \vu) = \mu(\vt) \left( \Grad \vu + \Grad^{\top} \vu -\frac{2}{3} (\Div \vu) \mathbb{I} \right) + \eta(\vt) (\Div \vu) \mathbb{I},
	\end{equation}
	with the shear viscosity $\mu=\mu(\vt)$ and the bulk viscosity $\eta=\eta(\vt)$ coefficients depending on temperature. Similarly, we suppose that the heat flux $\textbf{q}=\vq(\vt, \Grad \vt)$ is determined by Fourier's law,
	\begin{equation} \label{fourier law}
		\textbf{q}(\vt, \nabla_x \vt) = -\kappa(\vt) \Grad \vt,
	\end{equation}
	with the heat conductivity coefficient $\kappa=\kappa(\vt)$. Finally, $G=G(x)$ is a given potential, usually identified with the gravitational one. 
	
	\subsection{Perforated domain}
	
	We study the scaled Navier--Stokes--Fourier system \eqref{continuity equation}--\eqref{fourier law} on $(0,T) \times \Omega_{\ve}$, where the time $T>0$ can be chosen arbitrarily large while $\Omega_{\ve}$ denotes a domain perforated with many obstacles. More precisely, given $\Omega, \textup{U}\subset \mathbb{R}^3$ two bounded $C^{2,\nu}$-domains, $0<\nu<1$, we assume
	\begin{equation} \label{perforated domain}
		\Omega_{\ve}:= \Omega \setminus \bigcup_{n=1}^{N(\ve)} \overline{\textup{U}}_{\ve,n},
	\end{equation}
	where $\{ \textup{U}_{\ve,n} \}_{n=1}^{N(\ve)}$ denotes the family of obstacles given by
		\begin{equation} \label{hole}
			\textup{U}_{\ve, n} := x_{\ve,n} + \ve^{\alpha} \textup{U};
		\end{equation}
		in particular, we suppose that $\textup{U}_{\ve,n} \subset \subset B_{\ve,n}$, with 
		$$B_{\ve,n}:= B(x_{\ve,n}, \ve^{\alpha})$$ 
		denoting the ball centred at $x_{n,\ve}$ and radius $\ve^{\alpha}$, $\alpha >1$. Furthermore, we suppose that the balls $\{ B_{\ve,n} \}_{n=1}^{N(\ve)}$ have mutual distance $\ve$. 
	Specifically, defining 
	\begin{equation*}
		D_{\ve,n}:=  B\left(x_{\ve,n}, \ve^{\alpha}+ \frac{1}{2}  \ve\right)
	\end{equation*}
	we require that the balls $D_{\ve,n}$  are mutually disjoint. The latter condition gives an upper limit on the number of holes as 
	\begin{equation} \label{number holes}
		N(\ve) \simeq  \frac{3}{4\pi} |\Omega| \left( \ve^{\alpha}+\frac{1}{2} \ve\right)^{-3} \lesssim \ve^{-3}.
	\end{equation}
	Note, however, that we do not assume any periodicity for the distribution of the holes.
	\begin{figure}[htp]
		\centering
		\includegraphics[width=5cm]{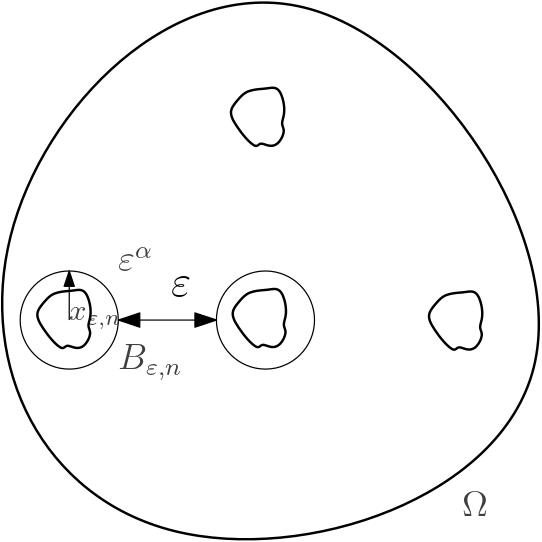}
		\caption{An example of perforated domain}
		\label{domain1}
	\end{figure}
	\begin{figure}[htp]
		\centering
		\includegraphics[width=12cm]{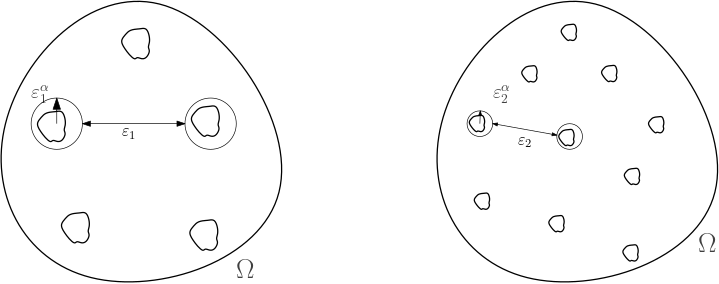}
		\caption{Perforated domains with $ \ve_1 > \ve_2$}
		\label{domain2}
	\end{figure}
	\par We consider the homogeneous Dirichlet and Neumann boundary conditions for the velocity $\vu$ and the temperature $\vt$, respectively; specifically,
	\begin{equation} \label{boundary conditions}
		\vu|_{\partial \Omega_{\ve}}=0, \quad \Grad\vt \cdot \textbf{n}|_{\partial \Omega_{\ve}}=0.
	\end{equation}

	\subsection{Constitutive relations}
	
	In order to motivate the existence of global-in-time weak solutions to system \eqref{continuity equation}--\eqref{fourier law} some extra assumptions are necessary. Motivated by \cite{FeiNov}, we assume 
	\begin{align}
		p(\vr, \vt) =p_{\textup{m}}(\vr,\vt)+ p_{\textup{rad}}(\vt), \quad &\mbox{with} \quad p_{\textup{m}}(\vr, \vt)= \vt^{\frac{5}{2}} P\left( \frac{\vr}{\vt^{\frac{3}{2}}} \right), \ p_{\textup{rad}}(\vt)=  \frac{a}{3} \vt^4, \label{pressure}\\
		e(\vr, \vt) =e_{\textup{m}}(\vr,\vt)+ e_{\textup{rad}}(\vr, \vt), \quad &\mbox{with} \quad e_{\textup{m}}(\vr, \vt)= \frac{3}{2}\frac{\vt^{\frac{5}{2}}}{\vr} P\left( \frac{\vr}{\vt^{\frac{3}{2}}} \right), \ e_{\textup{rad}}(\vr, \vt)=   \frac{a}{\vr} \vt^4, \label{internal energy}\\
		s(\vr, \vt) =s_{\textup{m}}(\vr,\vt)+ s_{\textup{rad}}(\vr, \vt), \quad &\mbox{with} \quad s_{\textup{m}}(\vr, \vt)= \mathcal{S} \left( \frac{\vr}{\vt^{\frac{3}{2}}} \right), \ s_{\textup{rad}}(\vr, \vt)=   \frac{4a}{3} \frac{\vt^3}{\vr}, \label{entropy}
	\end{align}
	where $a>0$, $P \in C^1[0,\infty) \cap C^3(0,\infty)$ satisfies
	\begin{equation} \label{capital P}
		P(0)=0, \quad P'(Z)>0 \mbox{ for } Z \geq 0, \quad 0< \frac{\frac{5}{3} P(Z)-P'(Z)Z}{Z} \leq c \mbox{ for } Z \geq 0, 
	\end{equation}
	and 
	\begin{equation} \label{capital s}
		\mathcal{S}'(Z) = - \frac{3}{2} \frac{\frac{5}{3} P(Z)-P'(Z)Z}{Z^2}.
	\end{equation}
	Consequently, the function $Z \mapsto P(Z)/ Z^{\frac{5}{3}}$ is decreasing and we assume
	\begin{equation} \label{constitutive relation}
		\lim_{Z \rightarrow \infty} \frac{P(Z)}{Z^{\frac{5}{3}}} = p_{\infty}>0.
	\end{equation} 
	Furthermore, we suppose that the transport coefficients $\mu$, $\eta$ and $\kappa$ are continuously differentiable functions of temperature $\vartheta$ satisfying
	\begin{align}
		0< \underline{\mu} (1+\vt) & \leq \mu(\vt) \leq \overline{\mu} (1+\vt), \label{shear condition}\\
		0 & \leq \eta(\vt) \leq \overline{\eta} (1+\vt), \label{bulk condition} \\
		0< \underline{\kappa} (1+\vt^3) & \leq \kappa(\vt) \leq \overline{\kappa} (1+\vt^3). \label{conductivity condition}
	\end{align}
	for all $\vt\geq 0$, with $\underline{\mu}, \overline{\mu}, \overline{\eta}, \underline{\kappa}, \overline{\kappa}$ positive constants. Finally, we suppose that the potential $G \in W^{1,\infty}(\Omega)$ has zero mean,
	\begin{equation} \label{zero mean potential}
		\int_{\Omega} G \ \dx =0.
	\end{equation}
	
	\subsection{Well-prepared initial data} \label{Well-prepared data}
	
	We suppose that
	\begin{equation} \label{initial data}
		\vr(0, \cdot) = \vr_{0,\ve}:= \overline{\vr} + \ve^m \vr_{0,\ve}^{(1)}, \quad \vu(0, \cdot) 
		= \vu_{0,\ve}, \quad \vt(0, \cdot) = \vt_{0,\ve}:= \overline{\vt} + \ve^m \vt_{0,\ve}^{(1)},
	\end{equation}
	where $\vr_{0,\ve}^{(1)}, \vu_{0,\ve}, \vt_{0,\ve}^{(1)}$ are measurable functions and $\overline{\vr}, \overline{\vt}$ are positive constants. Moreover, in order to get uniform bounds on $\Omega$ and to guarantee the extension of the field equations to the whole domain, we suppose that $[\vr_{0,\ve}^{(1)}, \vu_{0,\ve}, \vt_{0,\ve}^{(1)}]$ are extended by zero on $\Omega \setminus \Omega_{\ve}$; more precisely, we denote 
	\begin{equation}
		\widetilde{\vr}_{0,\ve}^{(1)}:= \begin{cases}
			\vr_{0,\ve}^{(1)} &\mbox{in } \Omega_{\ve}, \\
			0 &\mbox{in } \Omega \setminus \Omega_{\ve}, 
		\end{cases} \quad 
		\widetilde{\vu}_{0,\ve}:= \begin{cases}
			\vu_{0,\ve} &\mbox{in } \Omega_{\ve}, \\
			\textbf{0} &\mbox{in } \Omega \setminus \Omega_{\ve}, 
		\end{cases} \quad 
		\widetilde{\vt}_{0,\ve}^{(1)}:= \begin{cases}
			\vt_{0,\ve}^{(1)} &\mbox{in } \Omega_{\ve}, \\
			0 &\mbox{in } \Omega \setminus \Omega_{\ve},
		\end{cases} 
	\end{equation}
	and
	\begin{equation*}
		[ \widehat{\vr}_{0,\ve}, \widehat{\vt}_{0,\ve}] := [\overline{\vr}, \overline{\vt}] + \ve^m [\widetilde{\vr}_{0,\ve}^{(1)}, \widetilde{\vt}_{0,\ve}^{(1)}].
	\end{equation*}
	
	In addition, we suppose that
	\begin{equation} \label{integral initial density and temperature}
		\int_{\Omega} \widetilde{\vr}_{0, \ve}^{(1)} \ \dx=\int_{\Omega} \widetilde{\vt}_{0, \ve}^{(1)} \ \dx =0 \quad \mbox{for all } \ve>0,
	\end{equation}
	and 
	\begin{align}
		\widetilde{\vr}_{0,\ve}^{(1)} &\to \vr_{0}^{(1)} \quad \mbox{weakly-}(*) \mbox{ in } L^{\infty}(\Omega) \mbox{ and a.e. in } \Omega, \label{i1}\\
		\widetilde{\vu}_{0,\ve} &\to \vu_{0} \hspace{0.55cm} \mbox{weakly-}(*) \mbox{ in } L^{\infty}(\Omega; \mathbb{R}^3) \mbox{ and a.e. in } \Omega, \label{i2}\\
		\widetilde{\vt}_{0,\ve}^{(1)} &\to \vt_{0}^{(1)} \quad \mbox{weakly-}(*) \mbox{ in } L^{\infty}(\Omega) \mbox{ and a.e. in } \Omega; \label{convergence initial temperatures}
	\end{align}
	moreover, in order to get the maximal regularity for the dissipative solution of the target system, we suppose
	\begin{equation} \label{Slobodeckii condition}
		\widetilde{\vt}_0^{(1)} \in W^{2-\frac{2}{p},p}(\Omega) \quad \mbox{with} \quad p=\frac{5}{4}.
	\end{equation}
	
	Finally, we suppose that limiting initial data $\vr_{0}^{(1)}, \vt_{0}^{(1)}$ are \textit{well-prepared}, meaning that they satisfy the following relation:
	\begin{equation} \label{well-prepared data}
		\frac{\partial p(\overline{\vr}, \overline{\vt})}{\partial \vr} \vr_{0}^{(1)} + \frac{\partial p(\overline{\vr}, \overline{\vt})}{\partial \vt} \vt_{0}^{(1)} = \overline{\vr}G.
	\end{equation}
	We restrict ourselves to the consideration of well-prepared data only. The problem is more interesting for ill-prepared data, where the presence of acoustic waves play an important role in the analysis of singular limits. We wish to consider it in our future works. 
	
	\subsection{Target system}
	
	Our goal is to show that the low Mach number asymptotic limit on a perforated domain leads to the \textit{Oberbeck-Boussinesq approximation} 
	\begin{align}
		\Div \vU &=0, \label{OB 1} \\
		\overline{\vr} \left[ \DerTime \vU + (\vU \cdot \Grad) \vU \right] + \Grad \Pi - \mu(\overline{ \vt}) \Delta_x \vU&= -A \Theta \Grad G, \label{OB 3}\\
		\overline{ \vr} c_p \left[ \DerTime \Theta + \vU \cdot \Grad \Theta  \right] -  \kappa(\overline{ \vt}) \Delta_x \Theta &=  \overline{ \vt} A \Grad G \cdot \vU \label{OB 2}
	\end{align}
	on the homogenized domain. Here,  $\overline{ \vr}, \overline{ \vt}$ are the positive constants introduced in Section \ref{Well-prepared data} while the positive constant $A$ is defined as 
	\begin{equation} \label{OB 4}
		A:= \overline{ \vr} \ a(\overline{\vr}, \overline{\vt}),
	\end{equation}
	where $a$ denotes the coefficient of thermal extension given by 
	\begin{equation}
		a(\vr, \vt) := \frac{1}{\vr} \frac{\partial_{\vt} p}{\partial_{\vr}p}(\vr, \vt),
	\end{equation}
	and $c_p$ is the specific heat at constant pressure evaluated in $(\overline{ \vr}, \overline{ \vt})$,
	\begin{equation} 
		c_p := \frac{\partial e}{\partial \vt}(\overline{ \vr}, \overline{ \vt}) + a(\overline{ \vr}, \overline{ \vt}) \frac{\overline{\vt}}{\overline{ \vr}} \frac{\partial p}{\partial \vt} (\overline{ \vr}, \overline{ \vt}).
	\end{equation}
	Moreover, the functions $[\vU, \Theta]$ inherit the same boundary conditions of $[\vu, \theta^{(1)}]$; more precisely, we suppose 
	\begin{equation} \label{OB 7}
		\vU|_{\partial \Omega}=0, \quad \Grad \Theta \cdot \textbf{n}|_{\partial \Omega}=0.
	\end{equation}
	
	\begin{remark}
		We point out that if the couple $[\vU, \Theta]$ is a strong solution of system \eqref{OB 1}--\eqref{OB 7}, it is easy to check that $\Div \Grad^{\top} \vU= \Grad \Div \vU=0$.  Therefore, the viscosity term appearing in \eqref{OB 3} can be equivalently written as 
		\begin{equation*}
			\mu(\overline{ \vt}) \Div (\Grad \vU + \Grad^{\top} \vU);
		\end{equation*}
		the latter will be preferred when introducing the concept of dissipative solution, cf. Definition \ref{dissipative solution OB} below. 
	\end{remark}
	
	\subsection{Notation}
	
	To avoid confusion, we fix the notation that will be used throughout the paper. 
	
	Given two positive quantities $A, B$, we write
	\begin{itemize}
		\item $A \simeq B$ if there exist positive constants $c_1, c_2$ such that $c_1 A \leq B \leq c_2 A$;
		\item $A \lesssim B$ if there exists a positive constant $c$ such that $A \leq c B$.
	\end{itemize}
	
	Moreover, given $Q\subseteq \mathbb{R}^N$, $N\geq 1$, an open set, $X$ a Banach space and $M\geq 1$, we denote with
	\begin{itemize}
		\item $\mathcal{D}(Q; X)=C^{\infty}_c(Q; X)$ the space of functions belonging to $C^{\infty}(Q; X)$ and having compact support in $Q$;
		\item $\mathcal{D}'(Q; \mathbb{R}^M)= [C^{\infty}_c(Q; \mathbb{R}^M)]^*$ the space of distributions;
		\item $\mathcal{M}(Q; \mathbb{R}^M)= \big[\overline{C_c(Q; \mathbb{R}^M)}^{\| \cdot \|_{\infty}} \big]^*$ the space of vector-valued Radon measures. If $\Omega \subset \mathbb{R}^N$ is a bounded domain, then $\mathcal{M}(\overline{\Omega})= [C(\overline{\Omega})]^*$.
		\item $\mathcal{M}^+(Q)$ the space of positive Radon measures;
		\item $\mathcal{M}^+(Q; \mathbb{R}^{N\times N}_{\textup{sym}})$ the space of tensor--valued Radon measures $\mathfrak{R}$ such that $\mathfrak{R}: (\xi \otimes \xi) \in \mathcal{M}^+(Q)$ for all $\xi \in \mathbb{R}^d$, and with components $\mathfrak{R}_{i,j}=\mathfrak{R}_{j,i}$;
		\item $L^p(Q;X)$, with $1\leq p\leq \infty$, the Lebesgue space defined on $Q$ and ranging in $X$;
		\item $W^{k,p}(Q; \mathbb{R}^M)$, with $1\leq p\leq \infty$ and $k$ a positive integer, the Sobolev space defined on $Q$;
		\item $W^{s,p}(Q; \mathbb{R}^M)$, with $1\leq p\leq \infty$ and $s \in (0,1)$, the Sobolev-Slobodeckii space defined on $Q$.
	\end{itemize}
	
	\subsection*{Structure of the paper}
	
	The plan for the paper is as follows.
	\begin{itemize}
		\item In Section \ref{Concepts of solution and main result}, we recall the definition of weak solution for the Navier--Stokes--Fourier system, cf. Definition \ref{definition weak solution}, and provide the definition of dissipative solution for the Oberbeck-Boussinesq system, cf. Definition \ref{dissipative solution OB}. Subsequently, we state our main result, cf. Theorem \ref{main theorem}. 
		\item Section \ref{Preparation} is devoted to the extension of the state variables defined on the perforated domain $ \Omega_{\ve} $ to the whole domain $ \Omega $, and to the derivation of all the necessary uniform estimates.
		\item In Section \ref{Field equations on the homogenized domain}, we extend the validity of the field equations to the homogenized domain $ \Omega $.
		\item Section \ref{Convergence} is dedicated to the limit passage, leading to the concept of dissipative solution for the target system, cf. Proposition \ref{First result}.
		\item In Section \ref{Weak--strong uniqueness principle}, we prove the weak-strong uniqueness principle for the Oberbeck-Boussinesq system, cf. Theorem \ref{Weak-strong uniqueness}.
	\end{itemize}
	
	\section{Concepts of solution and main result} \label{Concepts of solution and main result}
	
	\subsection{Weak solution}
	
	We start providing the definition of weak solution to the Navier--Stokes-Fourier system, whose global-in-time existence was proved in \cite[Theorem 3.1]{FeiNov}.
	
	\begin{definition}[Weak solution of the Navier--Stokes--Fourier system on perforated domains] \label{definition weak solution}
		Let $\Omega\subset \mathbb{R}^3$ be a bounded $C^{2,\nu}$-domain. Moreover, let the thermodynamic variables $p$, $e$, $s$ satisfy hypotheses \eqref{pressure}--\eqref{constitutive relation} and the transport coefficients $\mu, \eta, \kappa$ satisfy conditions \eqref{shear condition}--\eqref{conductivity condition}. For any fixed $\ve >0$, we say that the trio of functions $[\vr_{ \ve}, \vu_{\ve}, \vt_{\ve}]$ such that
		\begin{align*}
			\vr_{\ve} &\in C_{\rm weak}([0,T]; L^{\frac{5}{3}}(\Omega_{\ve})), \\
			\vu_{\ve} &\in L^2(0,T; W_0^{1,2}(\Omega_{\ve}; \mathbb{R}^3)), \\
			\vr_\ve \vu_\ve &  \in  C_{\rm weak}([0,T]; L^{\frac{5}{4}}(\Omega_{\ve}; \mathbb{R}^3)),\\
			\vt_{\ve}  &\in L^{\infty}(0,T; L^4(\Omega_{\ve})),\\
			(\log \vt_{\ve}, \vt_{\ve}^{\beta}) & \in L^2(0,T; W^{1,2}(\Omega_{\ve}; \mathbb{R}^2)) \quad \mbox{for any } 1\leq \beta \leq \frac{3}{2},
		\end{align*}
		
		is a \emph{weak solution} of the scaled Navier--Stokes--Fourier system \eqref{continuity equation}--\eqref{fourier law} in $(0,T) \times \Omega_{\ve}$, where $\Omega_{\ve}$ is the perforated domain given by \eqref{perforated domain}, \eqref{hole}, with the boundary conditions \eqref{boundary conditions} and initial data \eqref{initial data} if the following holds.
		\begin{itemize}
			\item[(i)] \textit{Weak formulation of the continuity equation.} The integral identity
			\begin{equation} \label{WF1}
				\left[\int_{\Omega_{\ve}} \vr_{\ve} \varphi(t, \cdot) \ \dx \right]_{t=0}^{t=\tau}= \int_{0}^{\tau} \int_{\Omega_{\ve}} [\vr_{\ve} \DerTime \vphi + \vr_{\ve} \vu_{\ve} \cdot \Grad \vphi] \ \dx \dt
			\end{equation}
			holds for any $\tau \in [0,T]$ and any $\vphi \in C^1([0,T] \times \overline{\Omega}_{\ve})$, with 
			\begin{equation*}
				\vr_{ \ve}(0,\cdot)=\vr_{ 0,\ve} \quad \mbox{a.e. in } \Omega_{\ve}.
			\end{equation*}
			\item[(ii)] \textit{Weak formulation of the renormalized continuity equation.} For any function
			\begin{equation*}
				b \in C^1[0, \infty),\ b' \in C_c[0, \infty)
			\end{equation*}
			the integral identity
			\begin{equation} \label{WF2}
				\begin{aligned}
					\left[\int_{\Omega_{\ve}} b(\vr_{\ve}) \vphi(t,\cdot) \dx\right]_{t=0}^{t=\tau}&= \int_{0}^{\tau} \int_{\Omega_{\ve}} \big[b(\vr_{\ve}) \DerTime\vphi+ b(\vr_{\ve})\vu_{\ve} \cdot \Grad \vphi  \big] \ \dx \textup{d}t \\
					& + \int_{0}^{\tau} \int_{\Omega_{\ve}}  \vphi \ \big( b(\vr_{\ve}) -b'(\vr_{\ve}) \vr_{\ve}  \big)\Div \vu_{\ve} \ \dx \dt
				\end{aligned}
			\end{equation}
			holds for any $\tau \in [0,T]$ and any $\vphi \in C^1([0,T] \times \overline{\Omega}_{\ve})$.
			\item[(iii)] \textit{Weak formulation of the momentum equation.} The integral identity
			\begin{equation} \label{WF3}
				\begin{aligned}
					\left[\int_{\Omega_{\ve}} \vr_{\ve}\vu_{\ve} \cdot \bfphi (t,\cdot) \dx\right]_{t=0}^{t=\tau} &= \int_{0}^{\tau} \int_{\Omega_{\ve}} \left[\vr_{\ve} \vu_{\ve} \cdot \DerTime \bfphi + [(\vr_{\ve} \vu_{\ve} \otimes \vu_{\ve}) -\mathbb{S}(\vt_{\ve}, \Grad \vu_{\ve})] : \Grad \bfphi\right] \dx \dt \\
					&+\frac{1}{\ve^{m}} \int_{0}^{\tau} \int_{\Omega_{\ve}} \left( \frac{1}{\ve^{m}} p(\vr_{\ve}, \vt_{\ve})\Div \bfphi + \vr_{\ve} \Grad G \cdot \bfphi \right) \dx \dt
				\end{aligned}
			\end{equation}
			holds for any $\tau \in [0,T]$ and any $\bfphi \in C^1( [0,T] \times \overline{\Omega}_{\ve}; \mathbb{R}^3)$, $\bfphi|_{\partial \Omega_{\ve}}=0$, with
			\begin{equation*}
				(\vr_{\ve} \vu_{\ve})(0,\cdot)= \vr_{0, \ve}\vu_{0,\ve} \quad \mbox{a.e. in } \Omega_{\ve}.
			\end{equation*}
			\item[(iv)] \textit{Weak formulation of the entropy equality.} There exists a non-negative measure
			\begin{equation*}
				\mathfrak{S}_{\ve} \in \mathcal{M} ([0,T] \times \overline{\Omega}_{\ve}),
			\end{equation*}
			such that the integral identity
			\begin{equation} \label{WF4}
				\begin{aligned}
					- \int_{\Omega_{\ve}} &\vr_{0,\ve} s(\vr_{0,\ve}, \vt_{0,\ve}) \vphi (0,\cdot) \dx \\
					& = \int_{0}^{T} \int_{\Omega_{\ve}} \left[\vr_{\ve} s(\vr_{\ve}, \vt_{\ve}) \big( \DerTime \vphi + \vu_{\ve} \cdot \nabla_x \vphi \big) - \frac{\kappa(\vt_{\ve})}{\vt_{\ve}} \Grad \vt_{\ve} \cdot \nabla_x \vphi\right] \dx\dt \\
					& +  \ve^{2m} \int_{0}^{T} \int_{\Omega_{\ve}} \frac{ \vphi}{\vt_{\ve}} \left[  \vS(\vt_{\ve}, \Grad \vu_{\ve}): \nabla_x \vu_{\ve} + \frac{1}{\ve^{2m}} \frac{\kappa(\vt_{\ve})}{\vt_{\ve}} |\Grad \vt_{\ve}|^2 \right] \dx\dt + \int_{0}^{T} \int_{\overline{\Omega}_{\ve}} \varphi \ \textup{d}\mathfrak{S}_{\ve},
				\end{aligned}
			\end{equation}
			holds for any $\vphi \in C_c^1([0,T) \times \overline{\Omega}_{\ve})$.
			\item[(v)] \textit{Energy equality.} The integral identity 
			\begin{equation} \label{WF5}
				\begin{aligned}
					\int_{\Omega_{\ve}} &\left( \frac{\ve^{2m}}{2} \vr_{\ve} |\vu_{\ve}|^2 + \vr_{\ve} e(\vr_{ \ve}, \vt_{\ve}) - \ve^{m} \vr_{\ve} G\right) (t,\cdot) \ \dx \\
					&=  \int_{\Omega_{\ve}} \left( \frac{\ve^{2m}}{2} \vr_{0,\ve} |\vu_{0,\ve}|^2 + \vr_{0,\ve} e(\vr_{ 0,\ve}, \vt_{0,\ve}) - \ve^{m} \vr_{0,\ve} G\right)  \dx
				\end{aligned}
			\end{equation}
			holds for a.e. $t\in (0,T)$.
		\end{itemize}
	\end{definition}
	
	\begin{remark}
		Even if we are dealing with functions defined only almost everywhere on $(0,T)$, the left-hand sides of equations \eqref{WF1}--\eqref{WF3} are well-defined since the density $\vr_{\ve}$ and the momentum $\vm_{\ve}=\vr_{\ve} \vu_{\ve}$ are weakly continuous in time.
	\end{remark}
	
		\begin{remark}
			We point out that in \cite{FeiNov} the authors introduce in the weak formulation of the entropy equality a positive measure $\sigma_{\varepsilon}$ such that 
			\begin{equation*}
				\sigma_{\varepsilon} \geq \frac{1}{\vartheta_{\varepsilon}} \left[ \ve^{2m} \ \vS(\vt_{\ve}, \Grad \vu_{\ve}): \nabla_x \vu_{\ve} + \frac{\kappa(\vt_{\ve})}{\vt_{\ve}} |\Grad \vt_{\ve}|^2 \right],
			\end{equation*}
			cf. \cite[Theorem 5.1]{FeiNov}. We have simply summed and subtracted the integral 
			\begin{equation*}
				\int_{0}^{T} \int_{\Omega_{\ve}} \frac{ \vphi}{\vt_{\ve}} \left[ \ve^{2m} \  \vS(\vt_{\ve}, \Grad \vu_{\ve}): \nabla_x \vu_{\ve} + \frac{\kappa(\vt_{\ve})}{\vt_{\ve}} |\Grad \vt_{\ve}|^2 \right] \dx\dt
			\end{equation*}
			in \eqref{WF4}, so that in our context the latter contains a new measure $\mathfrak{S}_{\ve}$ such that
			\begin{equation*}
				\mathfrak{S}_{\ve}:= \sigma_{\varepsilon}-  \frac{1}{\vartheta_{\varepsilon}} \left[ \ve^{2m} \ \vS(\vt_{\ve}, \Grad \vu_{\ve}): \nabla_x \vu_{\ve} + \frac{\kappa(\vt_{\ve})}{\vt_{\ve}} |\Grad \vt_{\ve}|^2 \right] \geq 0.
			\end{equation*}
		\end{remark}
	
	\begin{remark}
		Defining the \textit{Helmholtz function} $H_{\overline{\vt}}= H_{\overline{\vt}}(\vr, \vt)$ as 
		\begin{equation*}
			H_{\overline{\vt}}(\vr, \vt) := \vr \big( e(\vr, \vt)-\overline{\vt}s(\vr, \vt) \big),
		\end{equation*}
		combining \eqref{integral initial density and temperature}, \eqref{WF4} and \eqref{WF5}, it is easy to show that the integral equality
		\begin{equation} \label{energy inequality with Helmholtz}
			\begin{aligned}
				\int_{\Omega_{\ve}} &\left[ \frac{1}{2} \vr_{\ve}|\vu_{\ve}|^2 + \frac{1}{\ve^{2m}} \left( H_{\overline{\vt}}(\vr_{\ve},\vt_{\ve})- (\vr_{\ve}-\overline{\vr}) \frac{\partial H_{\overline{\vt}}(\overline{\vr}, \overline{\vt})}{\partial \vr} -H_{\overline{\vt}}(\overline{\vr}, \overline{\vt}) \right) - \frac{\vr_{\ve}-\overline{\vr}}{\ve^m}  G \right] (\tau,\cdot) \ \dx \\
				+ &\int_{0}^{\tau} \int_{\Omega_{\ve}} \frac{\overline{\vt}}{\vt_{\ve}} \left( \vS(\vt_{\ve}, \Grad \vu_{\ve}): \nabla_x \vu_{\ve} + \frac{1}{\ve^{2m}}  \frac{\kappa(\vt_{\ve})}{\vt_{\ve}} |\Grad \vt_{\ve}|^2 \right) \dx\dt + \frac{\overline{\vt}}{\ve^{2m}} \mathfrak{S}_{\ve} ([0,\tau] \times \overline{\Omega}_{\ve}) \\
				= \int_{\Omega_{\ve}} &\left[ \frac{1}{2} \vr_{0,\ve}|\vu_{0,\ve}|^2 + \frac{1}{\ve^{2m}} \left( H_{\overline{\vt}}(\vr_{0,\ve},\vt_{0,\ve})- (\vr_{0,\ve}-\overline{\vr}) \frac{\partial H_{\overline{\vt}}(\overline{\vr}, \overline{\vt})}{\partial \vr} -H_{\overline{\vt}}(\overline{\vr}, \overline{\vt}) \right)  - \frac{\vr_{0,\ve}-\overline{\vr}}{\ve^m}  G \right] \dx 
			\end{aligned}
		\end{equation}
		holds for a.e. $\tau \in (0,T)$, where $\overline{\vr}, \overline{\vt}$ are the positive constants appearing in the definition of the initial density and temperature in \eqref{initial data}. 
	\end{remark}
	
	\subsection{Dissipative solution}
	
	Inspired by \cite{AbbFei1}, we will refer to the concept of dissipative solutions, i.e. solutions that satisfy the target system in the weak sense but with extra defect terms appearing in the equations and in the energy inequality. The motivation of the following definition will be clarified in the proof of Proposition \ref{First result}, when performing the passage to the limit.
	
	\begin{definition}[Dissipative solution of the Oberbeck-Boussinesq system] \label{dissipative solution OB}
		Let $\Omega\subset \mathbb{R}^3$ be a bounded $C^{2,\nu}$-domain. We say that the couple of functions
		\begin{equation} \label{regularity class}
			\begin{aligned}
				\vu  &\in C_{\rm weak}([0,T]; L^2(\Omega; \mathbb{R}^3)) \cap L^2(0,T; W^{1,2}_0(\Omega; \mathbb{R}^3)), \\
				\vt^{(1)}  &\in C([0,T]; W^{2-\frac{2}{p},p}(\Omega)) \cap W^{1,p}(0,T; L^p(\Omega)) \cap L^p(0,T; W^{2,p}(\Omega)), \quad  p=\frac{5}{4}
			\end{aligned}
		\end{equation}
		is a \textit{dissipative solution} of the Oberbeck-Boussinesq system \eqref{OB 1}--\eqref{OB 7} in $[0,T] \times \Omega$ with initial data $[\vu_0, \vt_{0}^{(1)}]$ if the following holds. 
		\begin{itemize}
			\item[(i)] \textit{Incompressibility.} Equation \eqref{OB 1} holds a.e. on $(0,T) \times \Omega$ for $\vU=\vu$.
			\item[(ii)] \textit{Incompressible Navier--Stokes system.} There exists a positive measure 
			\begin{equation*}
				\mathfrak{R} \in L^{\infty}(0,T; \mathcal{M}^{+}(\overline{\Omega}; \mathbb{R}^{3\times 3}_{\rm sym}))
			\end{equation*}
			such that the integral identity
			\begin{equation} \label{DF OB 1}
				\begin{aligned}
					\overline{\vr} \left[\int_{\Omega} \vu \cdot \bfphi (t,\cdot) \dx\right]_{t=0}^{t=\tau} &= \overline{\vr}\int_{0}^{\tau} \int_{\Omega} \left[\vu\cdot \DerTime \bfphi -  (\vu \cdot \Grad) \vu \cdot  \bfphi \right] \dx \dt \\
					&- \int_{0}^{\tau} \int_{\Omega} \left[ \mu(\overline{ \vt}) (\Grad \vu + \Grad^{\top} \vu) : \Grad \bm{\varphi} +A\vt^{(1)}\Grad G \cdot \bfphi \right] \dx \dt \\
					&+ \int_{0}^{\tau} \int_{\overline{\Omega}} \Grad \bm{\varphi}: \textup{d} \mathfrak{R} \ \dt
				\end{aligned}
			\end{equation}
			holds for any $\tau \in [0,T]$ and any $\bm{\varphi} \in C^1 ([0,T] \times \overline{\Omega}; \mathbb{R}^3)$, $\bm{\varphi}|_{\partial \Omega}=0$ such that $\Div \bm{\varphi}=0$, with
			\begin{equation*}
				\vu(0,\cdot)=\vu_{0} \quad \mbox{a.e. in } \Omega.
			\end{equation*}
			\item[(iii)] \textit{Heat equation with insulated boundary}. Equation \eqref{OB 2} holds a.e. on $(0,T) \times \Omega$ for $\Theta= \vt^{(1)}$, $\vU=\vu$, with $\Grad \vt^{(1)} \cdot \textbf{n}|_{\partial \Omega}=0$ in the sense of traces and $\vt^{(1)}(0,\cdot)= \vt_0^{(1)}$ a.e. in $\Omega$.
			\item[(iv)] \textit{Energy inequality.} There exists a positive measure 
			\begin{equation*}
				\mathfrak{E} \in L^{\infty}(0,T; \mathcal{M}^+(\overline{\Omega}))
			\end{equation*}
			such that the integral inequality 
			\begin{equation} \label{DF OB 2}
				\begin{aligned}
					&\int_{\Omega} \left( \frac{1}{2} \overline{ \vr}|\vu|^2+ \frac{c_p}{2} \frac{\overline{ \vr}}{\overline{ \vt}}  \big| \vt^{(1)} \big|^2  \right) (\tau, \cdot) \dx + \int_{\Omega} \textup{d} \mathfrak{E}(\tau) \\
					+ \frac{\mu(\overline{ \vt})}{2} &\int_{0}^{\tau}  \int_{\Omega}|\Grad \vu+ \Grad^{\top} \vu |^2 \ \dx \dt + \frac{\kappa(\overline{\vt})}{\overline{ \vt}} \int_{0}^{\tau}  \int_{\Omega}   |\Grad \vt^{(1)} |^2 \ \dx \dt \\
					\leq &\int_{\Omega} \left( \frac{1}{2} \overline{ \vr}|\vu_{0}|^2+ \frac{c_p}{2} \frac{\overline{ \vr}}{\overline{ \vt}}  \big| \vt_0^{(1)} \big|^2 \right) \dx
				\end{aligned}
			\end{equation}
			holds for a.e. $\tau \in (0,T)$.
			\item[(v)] \textit{Compatibility condition.} There holds
			\begin{equation} \label{compatibility condition}
				\trace [\mathfrak{R}] \simeq \mathfrak{E}.
			\end{equation}
		\end{itemize}
	\end{definition}
	
	\subsection{Main result}
	
	Having collected all the necessary ingredients, we are now ready to state our main result.
	
	\begin{theorem} \label{main theorem}
		Let 
		\begin{itemize}
			\item[-] the constants $\alpha$ and $m$ be fixed such that 
			\begin{equation} \label{hypothesis parameters}
				3< \alpha < m;
			\end{equation}
			\item[-] $\Omega\subset \mathbb{R}^3$ be a bounded $C^{2,\nu}$-domain and $\{\Omega_{\ve}\}_{\ve >0}$ be a family of perforated domains defined by \eqref{perforated domain}, \eqref{hole}; 
			\item[-] the thermodynamic variables $p$, $e$, $s$ satisfy hypotheses \eqref{pressure}--\eqref{constitutive relation}; 
			\item[-] the transport coefficients $\mu, \ \eta, \ \kappa$ satisfy conditions \eqref{shear condition}--\eqref{conductivity condition}; 
			\item[-] the potential $G$ have zero mean \eqref{zero mean potential}; 
			\item[-] $\{[\vr_{ 0,\ve}, \vu_{0,\ve}, \vt_{0, \ve}]\}_{\ve >0}$ be a family of initial data satisfying conditions \eqref{initial data}--\eqref{well-prepared data}.
		\end{itemize}
		
		Moreover, let 
		\begin{itemize}
			\item[-] $\{ [\vr_{ \ve}, \vu_\ve, \vt_{ \ve}] \}_{\ve >0}$ be the family of weak solutions to the scaled Navier--Stokes--Fourier system on the perforated domains, emanating from $\{[\vr_{ 0,\ve}, \vu_{0,\ve}, \vt_{0, \ve}]\}_{\ve >0}$ in the sense of Definition \ref{definition weak solution}; 
			\item[-] $\{[ \widehat{\vr}_{\ve}, \vuex, \widehat{\vt}_{\ve}]\}_{\ve >0}$ be the family of their extensions to the homogenized domain $\Omega$, specified in Section \ref{Extensions} below. 
		\end{itemize}
		
		Then there exists a positive time $T^*$ such that, passing to suitable subsequences as the case may be,
		\begin{align}
			\vuex \rightharpoonup \vU \quad &\mbox{in } L^2(0, T^*; W^{1,2}_0(\Omega; \mathbb{R}^3)), \label{m1}\\
			\frac{ \widehat{\vt}_{\ve}-\overline{ \vt}}{\ve^m} \rightharpoonup \Theta \quad &\mbox{in } L^2(0,T^*; W^{1,2}(\Omega)), \label{m2}
		\end{align}
		where $[\vU, \Theta]$ is the strong solution to the Oberbeck-Boussinesq system emanating from $[\vU_0, \Theta_0]= [\vu_0, \vt_{0}^{(1)}]$, with $\vu_0, \vt_{0}^{(1)}$ the weak limits appearing in \eqref{i2}, \eqref{convergence initial temperatures}, respectively.
	\end{theorem}
	
	\begin{remark}
		The positive time $T^*$ appearing in \eqref{m1}, \eqref{m2} denotes the maximal time of existence of strong solution to the Oberbeck-Boussinesq system \eqref{OB 1}--\eqref{OB 7}, cf. Theorem \ref{Existence strong solutions OB}.
	\end{remark}
	
		\begin{remark}
			Few comments regarding the optimality of the assumption \eqref{hypothesis parameters} are in order. Due to the fact that the $p$-capacity of the union of all the holes in $\Omega_{\ve}$ is approximately corresponding to $\ve^{(3-p)\alpha-3}$ (see \cite[Remark 2.4]{Lu}), combined with the low integrability of the pressure terms $p_{\ve}^{(1)}$, we must conclude that estimates \eqref{p1} and \eqref{p2} deduced in the proof of Lemma \ref{Entropy inequality ext} below are sharp. Therefore, conditions 
			\begin{equation*}
				\alpha-3 >0, \quad m-\alpha >0,
			\end{equation*}
			are necessary and cannot be improved. 
		\end{remark}
	
	
	Theorem \ref{main theorem} is a direct consequence of two results: first, we will show that the extended weak solutions of the Navier--Stokes-Fourier system converge to the dissipative solution of the Oberbeck--Boussinesq system, cf. Proposition \ref{First result}; secondly, by proving the weak--strong uniqueness principle, we are able to conclude that the dissipative solution must coincide with the strong solution of the target system, as long as the latter exists, cf. Theorem \ref{Weak-strong uniqueness}.

	\section{Preparation} \label{Preparation}
	\subsection{Extension of functions} \label{Extensions}
	
	In order to get the uniform bounds on the homogenized domain $\Omega$ and the correspondent convergences necessary to pass to the limit, we first need to properly extend all the quantities appearing in the system. 
	
	From now on, we will denote
	\begin{equation} \label{d1}
		\vr_{\ve}^{(1)}:= \frac{\vr_{ \ve}- \overline{\vr} }{\ve^{m}}, \quad \vt_{\ve}^{(1)}:= \frac{\vt_{ \ve}- \overline{\vt} }{\ve^{m}}, \quad \ell_{\ve}^{(1)}:= \frac{\log(\vt_{ \ve})- \log(\overline{\vt})}{\ve^m}.
	\end{equation}
	We can simply extend $\big[\vr_{\ve}^{(1)}, \vu_{\ve} \big]$ by zero on $\Omega \setminus \Omega_{\ve}$; more precisely, we consider
	\begin{align}\label{extn-rho,u}
		\widetilde{\vr}_{\ve}^{(1)} :=
		\begin{cases}
			\vr_{\ve}^{(1)}         & \text{in } \Omega_{\ve}  \\
			0        &\text{in } \Omega\setminus \Omega_{\ve} 
		\end{cases},\quad 
		\widetilde{\vu}_\ve :=
		\begin{cases}
			\vu_\ve     &   \text{in } \Omega_{\ve}  \\
			\textbf{0}       &\text{in } \Omega \setminus\Omega_{\ve} 
		\end{cases}.
	\end{align}
	The extension of $\vt_{ \ve}^{(1)}$ and $\ell_{\ve}^{(1)}$ is more delicate due to the Neumann boundary condition for the temperature: the extension by zero may not preserve the $W^{1,2}$-regularity. However, we may use the spatial extension $E_{\ve}$ constructed in \cite[Lemma 4.1]{LuPok}.
		\begin{lemma} \label{operator E}
			Suppose $ \Omega_{\ve} $ is given by \eqref{perforated domain}, \eqref{hole}. For any $\ve \in (0,1)$, there exists an extension operator
			\begin{equation*}
				E_\ve : W^{1,2}(\Omega_{\ve}) \rightarrow W^{1,2}(\Omega)
			\end{equation*}
			such that for each $ \varphi \in W^{1,2}(\Omega_{\ve})  $ and any $1\leq q\leq \infty$ we have 
			\begin{align}
				E_{\ve}(\varphi) &= \varphi \quad \text{in } \Omega_{\ve},\\
				\Vert  E_{\ve}(\varphi)  \Vert_{W^{1,2}(\Omega)} &\leq c \Vert \varphi  \Vert_{ W^{1,2}(\Omega_{\ve})}, \label{e10} \\
				\Vert  E_{\ve}(\varphi)  \Vert_{L^q(\Omega)} &\leq c \Vert \varphi  \Vert_{ L^q(\Omega_{\ve})}, \label{e11}
			\end{align}
			where the positive constant $c$ is independent of $\ve$.
		\end{lemma}
		In view of the aforementioned Lemma, we define
	\begin{equation}
		\widehat{\vt}_{\ve}^{(1)}:= 
		\begin{cases}
			\vt_{\ve}^{(1)}      & \text{in } \Omega_{\ve}  \\
			E_\ve(\vt_{\ve}^{(1)} )       &\text{in } \Omega \setminus \Omega_{\ve} 
		\end{cases}, \quad 
		\widehat{\ell}_{\ve}^{(1)}:= 
		\begin{cases}
			\ell_{\ve}^{(1)}      & \text{in } \Omega_{\ve}  \\
			E_\ve(\ell_{\ve}^{(1)} )       &\text{in } \Omega \setminus \Omega_{\ve} 
		\end{cases}.
	\end{equation}
	Accordingly, we consider the following extensions 
	\begin{equation*}
		[ \widehat{\vr}_{\ve}, \widehat{\vt}_{\ve}, \widehat{\ell}_{\ve}] := [\overline{\vr} , \overline{\vt}, \log(\overline{\vt})] + \ve^m [\widetilde{\vr}_{\ve}^{(1)}, \widehat{\vt}_{\ve}^{(1)}, \widehat{\ell}_{\ve}^{(1)}].
	\end{equation*}
	Next, we introduce analogous quantities to \eqref{d1} for the thermodynamic functions, 
	\begin{equation*}
		p_{\ve}^{(1)}:= \frac{p(\vr_{\ve}, \vt_{\ve})- p(\overline{\vr}, \overline{\vt})}{\ve^m}, \quad e_{\ve}^{(1)}:= \frac{e(\vr_{\ve}, \vt_{\ve})- e(\overline{\vr}, \overline{\vt})}{\ve^m}, \quad s_{\ve}^{(1)}:= \frac{s(\vr_{\ve}, \vt_{\ve})- s(\overline{\vr}, \overline{\vt})}{\ve^m},
	\end{equation*}
	and for the heat conductivity coefficient,
	\begin{equation*}
		\kappa_{\ve}^{(1)}:= \frac{\kappa(\vt_{\ve})- \kappa(\overline{\vt})}{\ve^m},
	\end{equation*}
	extending them by zero on $\Omega \setminus \Omega_{\ve}$: 
	\begin{equation*}
		[\widetilde{p}_{\ve}^{(1)}, \widetilde{e}_{\ve}^{(1)}, \widetilde{s}_{\ve}^{(1)}, \widetilde{\kappa}_{\ve}^{(1)}]:=
		\begin{cases}
			[p_{\ve}^{(1)}, e_{\ve}^{(1)}, s_{\ve}^{(1)}, \kappa_{\ve}^{(1)}]     &   \text{in } \Omega_{\ve}  \\
			\textbf{0}       &\text{in } \Omega \setminus\Omega_{\ve} 
		\end{cases}.
	\end{equation*}
	Proceeding as before, we consider the following extensions 
	\begin{align}
		[\widehat{p}_{\ve}, \widehat{e}_{\ve}, \widehat{s}_{\ve}, \widehat{\kappa}_{\ve}]&:= [p(\overline{ \vr}, \overline{ \vt}), e(\overline{\vr}, \overline{\vt}), s(\overline{\vr}, \overline{\vt}), \kappa(\overline{\vt})]+ \ve^m [\widetilde{p}_{\ve}^{(1)}, \widetilde{e}_{\ve}^{(1)}, \widetilde{s}_{\ve}^{(1)}, \widetilde{\kappa}_{\ve}^{(1)}];
	\end{align}
	notice, in particular, that $\widehat{\vr}_{\ve}= \overline{\vr}$ and $[ \widehat{p}_{\ve}, \widehat{e}_{\ve}, \widehat{s}_{\ve}, \widehat{\kappa}_{\ve}]= [p(\overline{ \vr}, \overline{ \vt}), e(\overline{\vr}, \overline{\vt}), s(\overline{\vr}, \overline{\vt}), \kappa(\overline{\vt})]$ on $\Omega \setminus \Omega_{\ve}$. Finally, we let the non-negative measure $\mathfrak{S}_{\ve}$ to be zero in $\Omega \setminus \Omega_{\ve} $
	\begin{equation*}
		\widetilde{\mathfrak{S}}_{\ve} := \begin{cases}
			\mathfrak{S}_{\ve}  & \text{in } \Omega_{\ve}  \\
			0     &\text{in } \Omega \setminus \Omega_{\ve} 
		\end{cases}.
	\end{equation*}
	
	\subsection{Essential and residual parts}
	
	Following \cite{FeiNov}, we introduce the set of essential values $\mathcal{O}_{\rm ess} \subset (0,\infty)^2$ together with its residual counterpart $\mathcal{O}_{\rm res} \subset (0,\infty)^2$ as 
	\begin{align*}
		\mathcal{O}_{\rm ess}&:= \left\{ (\vr,\vt) \in \mathbb{R}^2 \  \Big| \  \frac{\overline{\vr}}{2} < \vr < 2\overline{\vr},\ \frac{\overline{\vt}}{2} < \vt < 2\overline{\vt}\right\}, \\
		\mathcal{O}_{\rm res}&:= (0, \infty)^2 \setminus \mathcal{O}_{\rm ess},
	\end{align*}
	while the essential set $\mathcal{M}_{\rm ess} \subset (0,T)  \times \Omega_{\ve}$ and its residual counterpart $\mathcal{M}_{\rm res} \subset (0,T)  \times \Omega_{\ve}$ are defined as 
	\begin{align*}
		\mathcal{M}_{\rm ess}&:=  \left\{ (t,x) \in (0,T) \times \Omega_{\ve} \ \big| \ \big( \vr_{\ve}(t,x), \vt_{\ve}(t,x) \big) \in \mathcal{O}_{\rm ess} \right\}, \\
		\mathcal{M}_{\rm res}&:= \big((0,T)  \times \Omega_{\ve} \big) \setminus \mathcal{M}_{\rm ess}.
	\end{align*}
	We point out that $\mathcal{O}_{\rm ess}, \mathcal{O}_{\rm res}$ are fixed subsets of $(0,\infty)^2$, while $\mathcal{M}_{\rm ess}, \mathcal{M}_{\rm res}$ are measurable subsets of the time-space cylinder $(0,T)  \times \Omega_{\ve}$ depending on $\vr_{\ve}, \vt_{\ve}$. Moreover, in view of the extensions introduced in section \ref{Extensions}, along with $\mathcal{M}_{\rm ess}, \mathcal{M}_{\rm res}$ it makes sense to consider a third set $\mathcal{M}_{\rm holes}$ defined as 
	\begin{equation*}
		\mathcal{M}_{\rm holes} := (0,T) \times (\Omega \setminus \Omega_\ve).
	\end{equation*}
	
	Denoting with $ h_{\ve, {\rm ex}}$ the extension of any measurable function $h_{\ve}$ defined on $(0,T) \times \Omega_\ve$, it makes sense to write 
	\begin{equation} \label{holes part}
		\begin{aligned}
			h_{\ve, {\rm ex}} &= h_{\ve} \mathbbm{1}_{(0,T) \times\Omega_{\ve}} + h_{\ve, {\rm ex}} \mathbbm{1}_{(0,T) \times(\Omega \setminus \Omega_{\ve})} := [h_{\ve}]_{\rm ess} + [h_{\ve}]_{\rm res} + [h_{\ve}]_{\rm holes }, \\
			[h_{\ve}]_{\rm ess}:= h_{\ve} \mathbbm{1}_{\mathcal{M}_{\rm ess}}, \quad &[h_{\ve}]_{\rm res}:= h_{\ve} \mathbbm{1}_{\mathcal{M}_{\rm res}}= h_{\ve} - [h_{\ve}]_{\rm ess}, \quad [h_{\ve}]_{\rm holes } : = h_{\ve, {\rm ex}} \mathbbm{1}_{\mathcal{M}_{\rm holes}} =  h_{\ve, {\rm ex}} - h_{\ve} \mathbbm{1}_{(0,T) \times\Omega_{\ve}}.
		\end{aligned}
	\end{equation}
	
	\subsection{Uniform bounds} 
	
	We are now ready to establish the uniform bounds on the whole domain $\Omega$.
	
	\begin{lemma}\label{uni-bound-Omega}
		Under the hypotheses of Theorem \ref{main theorem}, the following uniform bounds hold.
		\begin{align}
			\esssup_{t\in (0,T)} \left| \mathcal{M}_{\rm holes}(t) \right| &\leq c \ve^{3(\alpha-1)}, \label{c3.1}\\[0.1cm]
			\esssup_{t\in (0,T)} \vert \mathcal{M}_{\rm res}(t) \vert & \leq c \ve^{2m} , \label{e4}\\[0.1cm]
			\widetilde{\mathfrak{S}}_{\ve} ([0,T] \times \overline{\Omega}) &\leq c\ve^{2m}, \label{e13}\\[0.1cm]
			\esssup_{t\in (0,T)}   \left[\big\|[\vr_{\ve}(t)]_{\rm res}\big\|_{L^{\frac{5}{3}}(\Omega)}^{\frac{5}{3}}+ \big\|[\vt_{\ve}(t)]_{\rm res}\big\|_{L^4(\Omega)}^4\right]   & \leq c \ve^{2m} , \label{e5} \\
			\left\| \left(\left[ \vr_{\ve}^{(1)} \right]_{\rm ess}, \left[ \vt_{\ve}^{(1)} \right]_{\rm ess}\right) \right\|_{L^{\infty}(0,T; L^2(\Omega; \mathbb{R}^2))} &\leq c, \label{e1}\\[0.1cm]
			\| \vuex\|_{L^2(0,T; W^{1,2}_0(\Omega; \mathbb{R}^3))} &\leq c, \label{e2}\\[0.1cm]
			\| \sqrt{ \widehat{\vr}_{\ve}} \vuex \|_{L^{\infty}(0,T; L^2 (\Omega; \mathbb{R}^3))} &\leq c, \label{e9.1}\\[0.1cm]
			\| \widehat{\vr}_{\ve} \vuex \|_{L^{\infty}(0,T; L^{\frac{5}{4}} (\Omega; \mathbb{R}^3))} &\leq c, \label{e9}\\[0.1cm]
			\| \widehat{\vr}_{\ve} \vuex \otimes \vuex \|_{L^{1}(0,T; L^{\frac{15}{14}} (\Omega; \mathbb{R}^{3\times 3}))}  & \leq c, \label{e9.2}\\[0.1cm]
			\left\|\left( \widehat{\vt}^{(1)}, \widehat{\ell}_{\ve}^{(1)}\right) \right\|_{L^2(0,T; W^{1,2}(\Omega; \mathbb{R}^2))} &\leq c, \label{e3} \\[0.1cm]
			\left\| \left[\frac{  p(\vr_{\ve}, \vt_{\ve})}{\ve^m}\right]_{\rm res} \right\|_{L^{\infty}(0,T; L^1(\Omega))} & \leq c\ve^m, \label{e15} \\[0.1cm]
			\left\| \left[\frac{ \vr_{\ve}  s(\vr_{\ve}, \vt_{\ve})}{\ve^m}\right]_{\rm res} \right\|_{L^2(0,T; L^{\frac{30}{23}}(\Omega))} & \leq c, \label{e6} \\[0.1cm]
			\left\| \left[\frac{ \vr_{\ve}  s(\vr_{\ve}, \vt_{\ve})}{\ve^m}\right]_{\rm res} \vuex \right\|_{L^2(0,T; L^{\frac{30}{29}}(\Omega; \mathbb{R}^3))} & \leq c, \label{e7} \\[0.1cm]
			\left\| \left[ \frac{\kappa(\vt_{\ve})}{\vt_{\ve}} \Grad \left( \frac{\vt_{\ve}}{\ve^m} \right) \right]_{\rm res}  \right\|_{L^\frac{14}{13}(0,T; L^{\frac{14}{13}}(\Omega; \mathbb{R}^3))} &\leq c. \label{e12}
		\end{align}
	\end{lemma}
	\begin{proof}
		The uniform bounds \eqref{e4}--\eqref{e9.1} and \eqref{e6}--\eqref{e12} are a direct consequence of \cite[Proposition 5.1]{FeiNov} since all the involved quantities vanish on $\Omega \setminus \Omega_{\ve}$. Bound \eqref{c3.1} follows from \eqref{number holes}, while bound \eqref{e9} can be deduced combining \eqref{e5}, \eqref{e1} and \eqref{e9.1}, and similarly, estimate \eqref{e9.2} is a consequence of bounds \eqref{e5}--\eqref{e2}. Estimate \eqref{e3} can be deduced from \cite[Proposition 5.1, equations (5.52), (5.53)]{FeiNov}, namely
		\begin{equation} \label{e3.1}
			\left\|\left(\vt_{\ve}^{(1)}, \ell_{\ve}^{(1)}\right) \right\|_{L^2(0,T; W^{1,2}(\Omega_{\ve}; \mathbb{R}^2))} \leq c, 
		\end{equation}
		combined with estimate \eqref{e10}. Finally, estimate \eqref{e15} follows from bounds \eqref{e4}, \eqref{e5} and the fact that, from hypotheses \eqref{capital P} and \eqref{constitutive relation}, we have
			\begin{equation*}
				\left[\frac{p(\vr_{ \ve}, \vt_{ \ve})}{\ve^m}\right]_{\rm res} \leq c \left( \left[\frac{1}{\ve^m}\right]_{\rm res} + \left[ \frac{\vr^{\frac{5}{3}}}{\ve^m}\right]_{\rm res} + \left[ \frac{\vt_{ \ve}^4}{\ve^m}\right]_{\rm res}\right);
			\end{equation*}
			see \cite[Section 5.3.3]{FeiNov} for more details.
	\end{proof}

	\section{Field equations on the homogenized domain} \label{Field equations on the homogenized domain}
	
	Before passing to the limit, along with the extension of all the quantities appearing in the primitive system \eqref{continuity equation}--\eqref{balance internal energy}, it is also necessary to extend the validity of the integral identities of Definition \ref{definition weak solution} to arbitrary test functions defined on the whole domain $\Omega$: the latter is the purpose of this section.
	
	\subsection{Continuity equation}
	
	\begin{lemma}
		Under the hypotheses of Theorem \ref{main theorem}, the integral identity 
		\begin{equation} \label{WF1 ext}
			\left[\int_{\Omega} \widehat{\vr}_{\ve} \vphi(t,\cdot) \ \dx\right]_{t=0}^{t=\tau}  = \int_{0}^{\tau} \int_{\Omega} [ \widehat{\vr}_{\ve} \DerTime \vphi +  \widehat{\vr}_{\ve} \widetilde{\vu}_{\ve} \cdot \Grad \vphi] \ \dx \dt,
		\end{equation}
		holds for any $\tau \in [0,T]$ and any $\vphi \in C^1([0,T] \times \overline{\Omega})$.
	\end{lemma}
	\begin{proof}
		Let $\varphi \in C^1_c([0,T]\times \overline{\Omega})$; then $\varphi|_{\overline{\Omega}_{\ve}} \in C^1_c([0,T]\times \overline{\Omega}_{\ve})$ can be used as test function in the weak formulation of the continuity equation \eqref{WF1}, obtaining
		\begin{equation*}
			\left[\int_{\Omega_{\ve}} \vr_{\ve} \vphi(t,\cdot) \ \dx\right]_{t=0}^{t=\tau} + \overline{\vr}  \left[\int_{\Omega \setminus \Omega_{\ve}} \vphi(t,\cdot) \ \dx\right]_{t=0}^{t=\tau}= \int_{0}^{\tau} \int_{\Omega_{\ve}} [\vr_{\ve} \DerTime \vphi + \vr_{\ve} \vu_{\ve} \cdot \Grad \vphi] \ \dx \dt + \overline{\vr}\int_{0}^{\tau} \int_{\Omega \setminus \Omega_{\ve}} \DerTime \vphi \ \dx \dt.
		\end{equation*}
		Now, it is enough to use the fact that $ \widehat{\vr}_{\ve}=\widehat{\vr}_{0,\ve}= \overline{\vr}$ and $\widetilde{\vu}_{\ve}= \textbf{0}$ on $\Omega \setminus \Omega_{\ve}$ to get \eqref{WF1 ext}.
	\end{proof}
	
	\subsection{Momentum equation} \label{Ex momentum equation}
	
	The extension of the weak formulation of the balance of momentum \eqref{WF3} is delicate due to the fact that the latter holds for test functions that vanish on the boundary of the perforated domain $\Omega_{\ve}$. Therefore, given an arbitrary test function defined on $\Omega$, we need to apply a suitable restriction operator 
	\begin{equation*}
		\mathcal{R}_{\ve}: W_0^{1,p}(\Omega; \mathbb{R}^3) \rightarrow W_0^{1,p}(\Omega_{\ve}; \mathbb{R}^3),
	\end{equation*}
	preserving the ``divergence-free" property; in particular, we need the following result, which can be found in \cite[Theorem 2.1]{Lu}.
		\begin{proposition} \label{main result}
			Let $p\in \left( 1, \infty\right)$ be fixed and let $\Omega_\ve$ be the perforated domain defined by \eqref{perforated domain}, \eqref{hole}. For any $\ve \in (0,1)$, there exists a linear operator 
			\begin{equation*}
				\mathcal{R}_{\ve}: W_0^{1,p}(\Omega; \mathbb{R}^3) \rightarrow W_0^{1,p}(\Omega_{\ve}; \mathbb{R}^3)
			\end{equation*}
			such that for any $\bm{\varphi} \in W^{1,p}_0(\Omega; \mathbb{R}^3)$,  
			\begin{equation} \label{estimate R}
				\| \mathcal{R}_{\ve} (\bm{\varphi}) \|_{W^{1,p}_0(\Omega_{\ve}; \mathbb{R}^{3})} \leq c \left( 1 + \ve^{\frac{3(\alpha-1)}{p}- \alpha} \right) \| \bm{\varphi}\|_{W^{1,p}_0(\Omega; \mathbb{R}^3)},
			\end{equation}
			where the positive constant $c$ does not depend on $\ve$. Furthermore, if $\Div \bm{\varphi}=0$ then $\Div \mathcal{R}_{\ve}(\bm{\varphi})=0$.
		\end{proposition}
		
		The operator $\mathcal{R}_{\ve}$ can be constructed implementing the technique developed by Allaire for the case $p=2$, cf. \cite[Section 2.2]{All}. Notice that an analogous restriction operator was considered by Diening, Feireisl and Lu when constructing the inverse of the divergence operator on perforated domains, cf. \cite[equation (3.12)]{DieFeiLu}. It is worth to point out that, as we are considering $C^{2,\nu}$- domains, the additional condition $\frac{3}{2}<p<3$ imposed in \cite[Theorem 2.1]{Lu} is not necessary in this context, cf. \cite[Remark 1.1]{Lu}.
		
		Adapting the construction presented in \cite[Theorem 2.1]{Lu}, it can be shown that the function $\bm{\varphi}- \mathcal{R}_{\ve}(\bm{\varphi})$ vanishes everywhere with the exception of the disjoints sets $D_{\ve, n}$, $n=1, \dots, N(\ve)$. Therefore, if we fix $r\in (1,\infty)$  and $p \in \left[ \frac{3}{2}, \infty \right)$ with $p>r$, from H\"{o}lder's inequality, estimate \eqref{estimate R} and the fact that $|D_{\ve,n}| \simeq \ve^{3}$, we have for any $\bfphi \in C_c^{\infty}(\Omega; \mathbb{R}^3)$,
		\begin{equation} \label{estimate difference}
			\begin{aligned}
				\| \bfphi - \mathcal{R}_{\ve}(\bfphi) \|_{L^r(\Omega; \mathbb{R}^3)} &= \sum_{n=1}^{N(\ve)} \| \bfphi - \mathcal{R}_{\ve}(\bfphi) \|_{L^r(D_{\ve,n}; \mathbb{R}^3)} \leq |D_{\ve,n}|^{\frac{1}{r}-\frac{1}{p}} \sum_{n=1}^{N(\ve)} \| \bfphi - \mathcal{R}_{\ve}(\bfphi) \|_{L^p(D_{\ve,n}; \mathbb{R}^3)} \\
				&\lesssim \ve^{3 \left(\frac{1}{r}-\frac{1}{p}\right)} \| \bfphi - \mathcal{R}_{\ve}(\bfphi) \|_{L^p(\Omega; \mathbb{R}^3)} \\
				& \lesssim \ve^{3\left(\frac{1}{r}-\frac{1}{p}\right)} \left( 1+ \ve^{\frac{3(\alpha-1)}{p}- 1} \right) \| \bfphi \|_{W_0^{1, \frac{3p}{p+3}}(\Omega; \mathbb{R}^3)};
			\end{aligned}
		\end{equation}
		notice that in the last line we have also used the Sobolev embedding
		\begin{equation*}
			W_0^{1, \frac{3p}{p+3}}(\Omega) \hookrightarrow L^p(\Omega), \quad \mbox{for } p \in \left[ \frac{3}{2}, \infty \right).
		\end{equation*} 
		Similarly, for any fixed $r, q \in (1,\infty)$ with $q>r$, it is easy to obtain that for any $\bfphi \in C_c^{\infty}(\Omega; \mathbb{R}^3)$,
		\begin{equation} \label{estimate difference gradient}
			\left\| \Grad \big[\bfphi - \mathcal{R}_{\ve}(\bfphi)\big] \right\|_{L^r(\Omega_{\ve})} \lesssim \ve^{3\left(\frac{1}{r}-\frac{1}{q}\right)} \left( 1+ \ve^{\frac{3(\alpha-1)}{q}- \alpha} \right) \| \bfphi \|_{W_0^{1, q}(\Omega)}. 
		\end{equation}
		
		We are now ready to extend the validity of the weak formulation of the balance of momentum \eqref{WF3} to test functions defined on the whole domain $\Omega$.
		
	\begin{lemma}
		Under the hypotheses of Theorem \ref{main theorem}, the integral identity
		\begin{equation} \label{WF3 ext}
			\begin{aligned}
				\left[\int_{\Omega} \widehat{\vr}_{\ve}\vuex \cdot \bfphi (t,\cdot) \dx\right]_{t=0}^{t=\tau} &= \int_{0}^{\tau} \int_{\Omega} \left[ \widehat{\vr}_{\ve} \widetilde{\vu}_{\ve} \cdot \DerTime \bfphi + (\widehat{\vr}_{\ve} \widetilde{\vu}_{\ve} \otimes \widetilde{\vu}_{\ve}) : \Grad \bfphi\right] \dx \dt \\
				&- \int_{0}^{\tau} \int_{\Omega} \left(\mathbb{S}(\widehat{\vt}_{\ve}, \Grad \widetilde{\vu}_{\ve}): \Grad \bfphi- \widetilde{\vr}_{\ve}^{(1)} \Grad G \cdot \bfphi \right) \dx \dt + \langle \bm{r}_{1,\ve}, \bm{\varphi} \rangle_{\mathcal{M}, C}, 
			\end{aligned}
		\end{equation}
		holds for any $\tau \in [0,T]$ and any $\bfphi \in C^1([0,T] \times \overline{\Omega}; \mathbb{R}^3)$, $\bfphi|_{\partial \Omega}=0$ such that $\Div \bfphi =0$, where the residual measure $\bm{r}_{1,\ve} \in \mathcal{M}([0,T] \times \overline{\Omega}; \mathbb{R}^3)$ satisfies
		\begin{equation} \label{distribution 1}
			| \langle \bm{r}_{1,\ve}, \bm{\varphi} \rangle_{\mathcal{M}, C}| \lesssim  \ve^{\gamma_1} \| \bm{\varphi}\|_{W_0^{1,\infty}((0,T) \times\Omega; \mathbb{R}^3)},
		\end{equation}
		with $\gamma_1$ the positive exponent defined in \eqref{gamma 1} below.
	\end{lemma}
	\begin{proof}
		First, let $\psi \in C_c^{\infty}(0,T)$ and $\bm{\varphi} \in C_c^{\infty}(\Omega; \mathbb{R}^3)$ be such that $\Div \bm{\varphi}=0$; then, we can use $\psi \ \mathcal{R}_{\ve}(\bm{\varphi})$ as test function in the weak formulation of the balance of momentum  \eqref{WF3}, where $\mathcal{R}_{\ve}$ is the linear operator constructed in Proposition \ref{main result}, obtaining the following identity
			\begin{equation} \label{wfbm with R}
				\begin{aligned}
					\int_{0}^{T} \psi' \int_{\Omega_{\ve}}  \vr_{\ve} \vu_{\ve} \cdot \mathcal{R}_{\ve}(\bfphi) \ \dx\dt &+ \int_{0}^{T} \psi \int_{\Omega_{\ve}} \left[ (\vr_{\ve} \vu_{\ve} \otimes \vu_{\ve}) - \mathbb{S}( \vt_{\ve}, \Grad \vu_{\ve}) \right] : \Grad \mathcal{R}_{\ve}(\bfphi) \ \dx \dt \\
					&+  \frac{1}{\ve^m} \int_{0}^{T} \psi \int_{\Omega_{\ve}} \vr_{\ve} \Grad G \cdot \mathcal{R}_{\ve}(\bfphi) \ \dx \dt =0.
				\end{aligned}
			\end{equation}
			Notice, in particular, that the term involving the pressure vanishes due to the fact that $\Div [ \mathcal{R}_{\ve}(\bm{\varphi})]=0$. Moreover, using the fact that 
			\begin{equation*}
				\int_{\Omega_{\ve}} \Grad G \cdot \mathcal{R}_{\ve}(\bfphi) \ \dx= - \int_{\Omega_{\ve}} G \ \Div [ \mathcal{R}_{\ve}(\bm{\varphi})]  \ \dx =0, 
			\end{equation*}
			we can equivalently write the last term appearing in \eqref{wfbm with R} as 
			\begin{equation*}
				\frac{1}{\ve^m} \int_{\Omega_{\ve}} \vr_{\ve} \Grad G \cdot \mathcal{R}_{\ve}(\bfphi) \ \dx = \int_{\Omega_{\ve}} \frac{\vr_{\ve}- \overline{ \vr}}{\ve^m} \ \Grad G \cdot \mathcal{R}_{\ve}(\bfphi) \ \dx  = \int_{\Omega_{\ve}} \vr_{\ve}^{(1)} \Grad G \cdot \mathcal{R}_{\ve}(\bfphi) \ \dx. 
			\end{equation*}
			If we now sum and subtract from identity \eqref{wfbm with R} the quantity 
			\begin{equation*}
				\int_{0}^{T} \psi' \int_{\Omega} \widehat{\vr}_{\ve} \widetilde{\vu}_{\ve} \cdot \bfphi \ \dx\dt + \int_{0}^{T} \psi \int_{\Omega} \left[ ( \widehat{\vr}_{\ve}\widetilde{\vu}_{\ve} \otimes \widetilde{\vu}_{\ve}) : \Grad \bfphi - \mathbb{S}( \widehat{\vt}_{\ve}, \Grad \widetilde{\vu}_{\ve}) : \Grad \bfphi + \widetilde{\vr}_{\ve}^{(1)} \Grad G \cdot \bfphi \right] \dx \dt,
			\end{equation*}
			we get
		\begin{equation*}
			\int_{0}^{T} \psi' \int_{\Omega} \widehat{\vr}_{\ve} \widetilde{\vu}_{\ve} \cdot \bfphi \ \dx\dt + \int_{0}^{T} \psi \int_{\Omega} \left[ ( \widehat{\vr}_{\ve} \widetilde{\vu}_{\ve} \otimes \widetilde{\vu}_{\ve}) : \Grad \bfphi - \mathbb{S}(\widehat{\vt}_{\ve}, \Grad \widetilde{\vu}_{\ve}) : \Grad \bfphi + \widetilde{\vr}_{\ve}^{(1)} \Grad G \cdot \bfphi \right] \dx \dt = \sum_{k=1}^{4} I_{\ve,k}
		\end{equation*}
		where, from the decomposition $\vr_{\ve}= \overline{\vr} + \ve^{m} \vr_{\ve}^{(1)}$, we have 
		\begin{align*}
			I_{\ve,1}&:= \int_{0}^{T} \psi' \int_{\Omega_{\ve}} \vr_{\ve} \vu_{\ve} \cdot  \big[\bfphi - \mathcal{R}_{\ve}(\bfphi) \big]\ \dx \dt \\
			&= \overline{\vr} \int_{0}^{T} \psi' \int_{\Omega_{\ve}} \vu_\ve \cdot  \big[\bfphi - \mathcal{R}_{\ve}(\bfphi) \big]\ \dx \dt + \ve^{m} \int_{0}^{T} \psi' \int_{\Omega_{\ve}} \vr_{\ve}^{(1)}\vu_\ve \cdot  \big[\bfphi - \mathcal{R}_{\ve}(\bfphi) \big]\ \dx \dt= I_{\ve,1}^{(1)}+ I_{\ve,1}^{(2)}\\
			I_{\ve,2}&:= \int_{0}^{T} \psi \int_{\Omega_{\ve}} (\vr_{\ve} \vu_{\ve} \otimes \vu_{\ve}): \nabla_x \big[\bfphi - \mathcal{R}_{\ve}(\bfphi) \big]  \ \dx \dt \\
			&= \overline{\vr}\int_{0}^{T} \psi \int_{\Omega_{\ve}} ( \vu_{\ve} \otimes \vu_{\ve}): \nabla_x \big[\bfphi - \mathcal{R}_{\ve}(\bfphi) \big]  \ \dx \dt + \ve^m \int_{0}^{T} \psi \int_{\Omega_{\ve}} (\vr_{\ve}^{(1)} \vu_{\ve} \otimes \vu_{\ve}): \nabla_x \big[\bfphi - \mathcal{R}_{\ve}(\bfphi) \big]  \ \dx \dt \\
			&= I_{\ve,2}^{(1)}+ I_{\ve,2}^{(2)}, \\
			I_{\ve,3}&:= -\int_{0}^{T} \psi \int_{\Omega_{\ve}}  \left[\mu(\vt_{\ve}) \left( \Grad \vu_{\ve} + \Grad^{\top} \vu_{\ve} -\frac{2}{d} (\Div \vu_{\ve}) \mathbb{I} \right) + \eta(\vt_{\ve}) (\Div \vu_{\ve}) \mathbb{I} \right]: \nabla_x \big[\bfphi - \mathcal{R}_{\ve}(\bfphi) \big]  \  \dx \dt, \\
			I_{\ve,4}&:=  \int_{0}^{T} \psi \int_{\Omega_{\ve}} \vr_{\ve}^{(1)} \nabla_x G \cdot \big[\bfphi - \mathcal{R}_{\ve}(\bfphi) \big] \ \dx \dt.
		\end{align*}
		Moreover, using the fact that 
		\begin{align*}
			\mu(\vt_{ \ve}) &\leq \overline{\mu}\left( 1+ \overline{\vt} \right) + \overline{\mu} \ve^m \vt_{\ve}^{(1)}, \\
			\eta(\vt_{ \ve}) &\leq \overline{\eta}\left( 1+ \overline{\vt} \right) + \overline{\eta} \ve^m \vt_{\ve}^{(1)},
		\end{align*}
		we obtain
		\begin{align*}
			|I_{\ve,3}| &\leq \left( 1+ \overline{\vt} \right) \int_{0}^{T} |\psi| \int_{\Omega_{\ve}}  \left[\overline{\mu} \left| \Grad \vu_{\ve} + \Grad^{\top} \vu_{\ve} -\frac{2}{d} (\Div \vu_{\ve}) \mathbb{I} \right| + \overline{\eta} |\Div \vu_{\ve}|  \right] \left|\nabla_x \big[\bfphi - \mathcal{R}_{\ve}(\bfphi) \big]\right|  \  \dx \dt, \\
			&+ \ve^{m} \int_{0}^{T} |\psi| \int_{\Omega_{\ve}}  \left[\overline{\mu} |\vt_{\ve}^{(1)}| \left| \Grad \vu_{\ve} + \Grad^{\top} \vu_{\ve} -\frac{2}{d} (\Div \vu_{\ve}) \mathbb{I} \right| + \overline{\eta} |\vt_{\ve}^{(1)}| |\Div \vu_{\ve}|  \right] \left|\nabla_x \big[\bfphi - \mathcal{R}_{\ve}(\bfphi) \big]\right|  \  \dx \dt \\
			& = I_{\ve,3}^{(1)}+ I_{\ve,3}^{(2)},
		\end{align*}
		We can now use the uniform bounds established in Lemma \ref{uni-bound-Omega} and estimates \eqref{estimate difference}, \eqref{estimate difference gradient} to deduce
		\begin{align}
			|I_{\ve,1}^{(1)}| &\lesssim \| \psi' \|_{L^2(0,T)}\| \vuex \|_{L^2 (0,T; L^6(\Omega))} \| \bfphi - \mathcal{R}_{\ve}(\bfphi) \|_{L^{\frac{6}{5}}(\Omega_{\ve})} \lesssim \left[\ve^{3} + \ve^{\frac{3}{2}(\alpha-1)}\right] \| \bfphi \|_{W_0^{1, \frac{6}{5}}(\Omega)}; \label{A1} \\[0.1cm]
			|I_{\ve,1}^{(2)}| &\lesssim \ve^{m} \| \psi' \|_{L^2(0,T)}\| \vrex^{(1)}\vuex \|_{L^2 (0,T; L^{\frac{30}{23}}(\Omega))} \| \bfphi - \mathcal{R}_{\ve}(\bfphi) \|_{L^{\frac{30}{7}}(\Omega_{\ve})} \notag \\
			&\lesssim  \left[\ve^{m+ \frac{2}{5}} + \ve^{m+\frac{3}{10}(\alpha-3)}\right] \| \bfphi \|_{W_0^{1, \frac{30}{13}}(\Omega)}; \label{A2} \\
			|I_{\ve,2}^{(1)}| &\lesssim \| \psi \|_{L^{\infty}(0,T)} \| \vuex \otimes \vuex\|_{L^1(0,T; L^3(\Omega))} \left\| \Grad \big[\bfphi - \mathcal{R}_{\ve}(\bfphi) \big] \right\|_{L^{\frac{3}{2}}(\Omega_{\ve})} \lesssim \left[\ve^{\frac{1}{2}}  + \ve^{\frac{1}{2}(\alpha-2) } \right] \| \bfphi \|_{W_0^{1,2}(\Omega)}; \label{A3} \\
			|I_{\ve,2}^{(2)}| &\lesssim  \ve^{m} \| \psi \|_{L^{\infty}(0,T)} \| \vrex^{(1)} \vuex \otimes \vuex\|_{L^1(0,T; L^{\frac{15}{14}}(\Omega))} \left\| \Grad \big[\bfphi - \mathcal{R}_{\ve}(\bfphi) \big] \right\|_{L^{15}(\Omega_{\ve})} \notag \\
			&\lesssim  \left[\ve^{m+\frac{1}{10} }  + \ve^{m-\frac{9}{10}\alpha}\right] \| \bfphi \|_{W_0^{1, 30}(\Omega)}; \label{A4} \\
			|I_{\ve,3}^{(1)}| & \lesssim  \| \psi \|_{L^2(0,T)}  \| \nabla_x \vuex \|_{L^2(0,T; L^2(\Omega))} \left\| \Grad \big[\bfphi - \mathcal{R}_{\ve}(\bfphi) \big] \right\|_{L^2(\Omega_{\ve})} \notag \\
			&\lesssim \left[\ve^{\frac{3}{2} \frac{\alpha-3}{5\alpha-9}}  + \ve^{ \frac{1}{2} \frac{
					(2\alpha-3)(\alpha-3)}{5\alpha-9}} \right] \| \bfphi \|_{W_0^{1, p}(\Omega)} \quad \mbox{with} \quad p:= 2+ \frac{\alpha-3}{2\alpha-3};  \label{A5}\\
			|I_{\ve,3}^{(2)}| & \lesssim \ve^m\| \psi \|_{L^{\infty}(0,T)}  \| \widehat{\vt}^{(1)} \nabla_x \vuex \|_{L^1(0,T; L^{\frac{3}{2}}(\Omega))} \left\| \Grad \big[\bfphi - \mathcal{R}_{\ve}(\bfphi) \big] \right\|_{L^3(\Omega_{\ve})} \notag \\
			&\lesssim \left[\ve^{m + \frac{1}{2}}  + \ve^{m-\frac{1}{2}\alpha } \right] \| \bfphi \|_{W_0^{1, 6}(\Omega)}; \\
			|I_{\ve,4}| & \lesssim \| \psi \|_{L^1(0,T)} \| \vrex^{(1)} \|_{L^{\infty}(0,T; L^{\frac{5}{3}}(\Omega))} \| \bfphi - \mathcal{R}_{\ve}(\bfphi) \|_{L^{\frac{5}{2}}(\Omega_{\ve})} \lesssim\left[\ve^{\frac{3}{5}} + \ve^{\frac{3}{5}\alpha-1}\right] \| \bfphi \|_{W_0^{1, \frac{15}{8}}(\Omega)}. \label{A7}
		\end{align}
		Due to hypothesis \eqref{hypothesis parameters}, the exponent for $\ve$ is positive in \eqref{A1}-\eqref{A7}. Therefore, condition \eqref{distribution 1} is satisfied choosing
		\begin{equation} \label{gamma 1}
			\gamma_1:= \min \left\{ m-\frac{9}{10}\alpha, \ \frac{3}{2} \frac{\alpha-3}{5\alpha-9} \right\}.
		\end{equation}
		To conclude the proof, it is now enough to use a density argument; indeed, without loss of generality, we can consider test functions of the form $\psi(t)\bm{\varphi}(x)$ with $\psi \in C_c^{\infty}(0,T)$, $\bm{\varphi} \in C_c^{\infty}(\Omega; \mathbb{R}^3)$, and then use the density of the smooth and compactly supported functions in the $L^p$-spaces for any $p\in (1,\infty)$.
	\end{proof}

	\subsection{Entropy equality}
	
	\begin{lemma} \label{Entropy inequality ext}
		Under the hypotheses of Theorem \ref{main theorem}, the integral equality
		\begin{equation} \label{WF4 ext}
			\begin{aligned}
				- \int_{\Omega} &\widetilde{\vr}_{0,\ve} \frac{s( \widehat{\vr}_{0,\ve}, \widehat{\vt}_{0,\ve})- s(\overline{\vr}, \overline{\vt})}{\ve^m} \vphi (0,\cdot) \dx \\
				& = \int_{0}^{T} \int_{\Omega} \left[ \widehat{\vr}_{\ve} \widetilde{s}_{\ve}^{(1)} \big( \DerTime \vphi + \widetilde{\vu}_{\ve} \cdot \nabla_x \vphi \big) - \widehat{\kappa}_{\ve} \Grad \widehat{\ell}_{\ve}^{(1)} \cdot \nabla_x \vphi\right] \dx\dt \\
				& + \ve^{m}\int_{0}^{T} \int_{\Omega} \varphi \left( \widehat{\vt}_{\ve}^{-1}  \vS( \widehat{\vt}_{\ve}, \Grad \widetilde{\vu}_{\ve}): \nabla_x \widetilde{\vu}_{\ve} + \widehat{\kappa}_{\ve} |\Grad  \widehat{\ell}_{\ve}^{(1)}|^2 \right) \dx\dt + \frac{1}{\ve^m} \int_{0}^{T} \int_{\overline{\Omega}} \varphi \ \textup{d}\widetilde{\mathfrak{S}}_{\ve} + \langle \bm{r}_{2,\ve}, \varphi \rangle_{\mathcal{M}, C}
			\end{aligned}
		\end{equation}
		holds for any $\vphi \in C_c^1([0,T) \times \overline{\Omega})$, where the residual measure $\bm{r}_{2,\ve} \in \mathcal{M}([0,T] \times \overline{\Omega}; \mathbb{R}^3)$ satisfies
			\begin{equation} \label{distribution 3}
				| \langle \bm{r}_{2,\ve}, \bm{\varphi} \rangle_{\mathcal{M}, C}| \lesssim  \ve^{\gamma_2} \| \bm{\varphi}\|_{W_0^{1,\infty}((0,T) \times\Omega; \mathbb{R}^3)},
			\end{equation}
			with $\gamma_2$ the positive exponent defined in \eqref{gamma 3} below.
	\end{lemma}
	
	\begin{proof}
		Let $\varphi \in C^1_c([0,T)\times \overline{\Omega})$; then $\varphi|_{\overline{\Omega}_{\ve}} \in C^1_c([0,T)\times \overline{\Omega}_{\ve})$ can be used as test function in the weak formulation of the entropy equality \eqref{WF4}, obtaining
		\begin{equation*} 
			\begin{aligned}
				- \int_{\Omega} &\widetilde{\vr}_{0,\ve} \frac{s( \widehat{\vr}_{0,\ve}, \widehat{\vt}_{0,\ve})- s(\overline{\vr}, \overline{\vt})}{\ve^m} \vphi (0,\cdot) \dx \\
				& = \int_{0}^{T} \int_{\Omega} \left[ \widehat{\vr}_{\ve}\widetilde{s}_{\ve}^{(1)} \big( \DerTime \vphi + \widetilde{\vu}_{\ve} \cdot \nabla_x \vphi \big) - \widehat{\kappa}_{\ve} \Grad \widehat{\ell}_{\ve}^{(1)} \cdot \nabla_x \vphi\right] \dx\dt \\
				& + \ve^{m}\int_{0}^{T} \int_{\Omega} \varphi \left( \widehat{\vt}_{\ve}^{-1}  \vS( \widehat{\vt}_{\ve}, \Grad \widetilde{\vu}_{\ve}): \nabla_x \widetilde{\vu}_{\ve} + \widehat{\kappa}_{\ve} |\Grad  \widehat{\ell}_{\ve}^{(1)}|^2 \right) \dx\dt + \frac{1}{\ve^m} \int_{0}^{T} \int_{\overline{\Omega}} \varphi \ \textup{d}\widetilde{\mathfrak{S}}_{\ve} \\
				& + \kappa(\overline{\vt}) \left(\int_{0}^{T}\int_{\Omega \setminus \Omega_{\ve}}  \Grad E_{\ve}(\ell_{\ve}^{(1)}) \cdot \Grad \varphi \   \dx \dt + \ve^m \int_{0}^{T}\int_{\Omega \setminus \Omega_{\ve}} | \Grad E_{\ve}(\ell_{\ve}^{(1)})  |^2 \varphi  \ \dx \dt  \right).
			\end{aligned}
		\end{equation*}
		We now focus on the last two integrals; from \eqref{e3}, we have
			\begin{equation*}
				\left| \int_{0}^{T}\int_{\Omega \setminus \Omega_{\ve}}  \Grad E_{\ve}(\ell_{\ve}^{(1)}) \cdot \Grad \varphi \   \dx \dt \right| \lesssim \| \widetilde{\ell}_{\ve}^{(1)}\|_{L^2(0,T; W^{1,2}(\Omega))} \| \varphi\|_{W_0^{1,\infty}} | \Omega \setminus \Omega_\ve|^{\frac{1}{2}} \lesssim \ve^{\frac{3(\alpha-1)}{2}},
			\end{equation*} 
			\begin{equation*}
				\left| \int_{0}^{T}\int_{\Omega \setminus \Omega_{\ve}} | \Grad E_{\ve}(\ell_{\ve}^{(1)})  |^2 \varphi  \ \dx \dt \right| \lesssim  \| \widetilde{\ell}_{\ve}^{(1)}\|_{L^2(0,T; W^{1,2}(\Omega))}^2 \| \varphi\|_{W_0^{1,\infty}} \lesssim 1.
			\end{equation*}
			Therefore, condition \eqref{distribution 3} is verified choosing 
			\begin{equation} \label{gamma 3}
				\gamma_2:= \min \left\{ \frac{3(\alpha-1)}{2}, \ m \right\}.
			\end{equation}
	\end{proof}

	\subsection{Boussinesq relation} \label{Boussinesq relation}
	
	Along with \eqref{WF1}-\eqref{WF4}, when letting $\ve$ go to zero, we need to consider an additional integral identity known as \textit{Boussinesq relation}, obtained by multiplying \eqref{WF3} by $\ve^m$. Similarly to Section \ref{Ex momentum equation}, we have to solve the problem of vanishing test function on the boundary of $\Omega_{\ve}$. Since in this context the $L^{\infty}$-norms of the gradient of the test functions will be necessary, we cannot use the restriction $\mathcal{R}_{\ve}$ as we did in Section \ref{Ex momentum equation}. Instead, we multiply our arbitrary test function defined on the whole $\Omega$ by a suitable smooth and compactly supported function on $\Omega_{\ve}$, cf. equation \eqref{g epsilon}.
	
	\begin{lemma}
		Under the hypotheses of Theorem \ref{main theorem}, the integral equality 
		\begin{equation} \label{WF6}
			\int_{0}^{T} \int_{\Omega} \left(\widetilde{p}_{\ve}^{(1)} \Div \bm{\varphi} + \widehat{\vr}_{\ve} \Grad G \cdot \bm{\varphi} \right)\ \dx \dt =  \langle \bm{r}_{3,\ve}, \bm{\varphi} \rangle_{\mathcal{D}', \mathcal{D}}
		\end{equation}
		holds for any $\bm{\varphi} \in C^{\infty}_c((0,T) \times \Omega; \mathbb{R}^3)$, where the residual distribution $\bm{r}_{3,\ve} \in \mathcal{D}'((0,T)\times \Omega; \mathbb{R}^3)$ satisfies
		\begin{equation} \label{distribution 2}
			| \langle \bm{r}_{3,\ve}, \bm{\varphi} \rangle_{\mathcal{D}', \mathcal{D}}| \lesssim  \ve^{\gamma_3} \| \bm{\varphi}\|_{W_0^{1,\infty}((0,T) \times \Omega; \mathbb{R}^3)},
		\end{equation}
		with $\gamma_3$ the positive exponent defined in \eqref{gamma 2} below.
	\end{lemma}
	
	\begin{proof}
		We start noticing that it is not difficult to construct a cut-off function $\phi_{\ve,n}$ such that 
			\begin{equation}
				\phi_{\ve,n} \in C^{\infty}_c ( B_{\ve,n} ), \quad 0\leq \phi_{\ve, n}\leq 1, \quad \phi \big|_{\overline{U_{\ve,n}}}=1, 
		\end{equation}
		and for any $1\leq p \leq \infty$, 
		\begin{align}
			\| \phi_{\ve,n} \|_{L^p(\mathbb{R}^3)} & \lesssim   \ve^{\frac{3}{p} \alpha}, \label{phi 1}\\
			\| \Grad \phi_{\ve,n} \|_{L^p(\mathbb{R}^3; \mathbb{R}^3)} & \lesssim  \ \ve^{\left( \frac{3}{p}-1 \right)\alpha}. \label{phi 2}
		\end{align}
		Let us now consider the function
		\begin{equation} \label{g epsilon}
			g_{\ve}(x) := 1- \sum_{n=1}^{N(\ve)} \phi_{\ve,n} (x);
		\end{equation}
		clearly $g_{\ve} \in C^{\infty}_c(\Omega_{\ve})$, $0\leq g_{\ve} \leq 1$, and from \eqref{number holes} and estimates \eqref{phi 1}, \eqref{phi 2}, we can deduce that 
		\begin{align}
			\| 1-g_{\ve} \|_{L^p(\Omega)} &\leq N(\ve)^{\frac{1}{p}} \| \phi_{\ve,n} \|_{L^p(\mathbb{R}^3)}  \lesssim \ve^{\frac{3(\alpha-1)}{p}}, \label{g1}\\
			\| \Grad g_{\ve} \|_{L^p(\Omega; \mathbb{R}^3)} &\leq N(\ve)^{\frac{1}{p}} \| \Grad \phi_{\ve,n} \|_{L^p(\mathbb{R}^3; \mathbb{R}^3)}  \lesssim \ve^{\frac{3(\alpha-1)}{p}-\alpha}. \label{g2}
		\end{align}
		
		Let $\bm{\varphi} \in C_c^{\infty}((0,T) \times \Omega; \mathbb{R}^3)$. Then we can multiply \eqref{WF3} by $\ve^m$ and use 
		\begin{equation*}
			\bm{\varphi}_{\ve}(t,x) := g_{\ve}(x) \bm{\varphi}(t,x) 
		\end{equation*}
		as test function in the resulting integral identity, obtaining
			\begin{align*}
				&\int_{0}^{T} \int_{\Omega_{\ve}} p_{\ve}^{(1)} \left( \nabla_x g_{\ve} \cdot \bm{\varphi} + g_{\ve} \Div \bm{\varphi}\right) \dx \dt + \int_{0}^{T} \int_{\Omega_{\ve}} g_{\ve} \vr_{ \ve} \nabla_x G \cdot \bm{\varphi} \ \dx \dt \\
				+ \ve^m \int_{0}^{T} \int_{\overline{\Omega}_{\ve}}& \left(g_{\ve}\vr_{ \ve} \vu_{\ve} \cdot \DerTime \bm{\varphi} + \big[ (\vr_{ \ve} \vu_{\ve} \otimes \vu_{\ve}) - \vS (\vt_{ \ve}, \nabla_x \vu_{\ve}) \big]: ( \nabla_x g_{\ve} \otimes \bm{\varphi} + g_{\ve} \nabla_x \bm{\varphi}) \right) \dx \dt =0. 
			\end{align*}
			Summing and subtracting from the previous identity the quantity 
			\begin{equation*}
				\int_{0}^{T} \int_{\Omega} \left(\widetilde{p}_{\ve}^{(1)} \Div \bm{\varphi} +\widehat{\vr}_{\ve}\Grad G \cdot \bm{\varphi} \right)\ \dx \dt 
			\end{equation*}
			we can equivalently write
		\begin{equation*}
			\int_{0}^{T} \int_{\Omega} \left(\widetilde{p}_{\ve}^{(1)} \Div \bm{\varphi} + \widehat{\vr}_{\ve} \Grad G \cdot \bm{\varphi} \right)\ \dx \dt = \sum_{k=1}^{6} I_{\ve, k},
		\end{equation*}
		where 
		\begin{align*}
			I_{\ve, 1}&:=  \int_{0}^{T} \int_{\Omega_{\ve}} p_{\ve}^{(1)} \ [- \Grad g_{\ve} \cdot \bm{\varphi} + (1-g_{\ve}) \Div \bm{\varphi}]  \  \dx \dt\\
			&= \int_{0}^{T} \int_{\Omega_{\ve}} [p_{\ve}^{(1)}]_{\rm ess} \ [- \Grad g_{\ve} \cdot \bm{\varphi} + (1-g_{\ve}) \Div \bm{\varphi}]  \ \dx \dt \\
			&+ \int_{0}^{T} \int_{\Omega_{\ve}} [p_{\ve}^{(1)}]_{\rm res} \ [- \Grad g_{\ve} \cdot \bm{\varphi} + (1-g_{\ve}) \Div \bm{\varphi}]  \ \dx \dt =  I_{\ve, 1}^{(1)}+ I_{\ve, 1}^{(2)}, \\
			I_{\ve, 2}&:=  \int_{0}^{T} \int_{\Omega_{\ve}} (1-g_{\ve}) \vr_{\ve} \Grad G \cdot \bm{\varphi}  \ \dx \dt, \\
			I_{\ve, 3}&:=  \overline{\vr}\int_{0}^{T} \int_{\Omega \setminus \Omega_{\ve}} \Grad G \cdot  \bm{\varphi} \ \dx \dt, \\
			I_{\ve, 4}&:=- \ve^m  \int_{0}^{T}  \int_{\Omega_{\ve}} g_{\ve}\vr_{ \ve} \vu_{\ve} \cdot \DerTime 
			\bm{\varphi} \  \dx \dt, \\
			I_{\ve, 5}&:=- \ve^m  \int_{0}^{T} \int_{\Omega_{\ve}} (\vr_{ \ve} \vu_{\ve} \otimes \vu_{\ve}) : (\Grad g_{\ve} \otimes \bm{\varphi}+ g_{\ve} \Grad \bm{\varphi}) \ \dx \dt \\
			I_{\ve, 6}&:= \ve^m  \int_{0}^{T}  \int_{\Omega_{\ve}} \vS (\vt_{ \ve}, \Grad \vu_{\ve}) : (\Grad g_{\ve} \otimes \bm{\varphi}+ g_{\ve} \Grad \bm{\varphi}) \  \dx \dt. 
		\end{align*}
		We can now use the uniform bounds established in Lemma \ref{uni-bound-Omega} and estimates \eqref{g1}, \eqref{g2} to get the following bounds.
		\begin{align}
			|I_{\ve,1}^{(1)}| &\lesssim  \left\| [p_{\ve}^{(1)}]_{\rm ess} \right\|_{L^{\infty}(0,T; L^2(\Omega))} \left( \| \Grad g_{\ve} \|_{L^2(\Omega)}+ \| 1-g_{\ve}\|_{L^2(\Omega)} \right) \| \bm{\varphi}\|_{W_0^{1,\infty}} \lesssim \ve^{\frac{1}{2}(\alpha-3)} \| \bm{\varphi}\|_{W_0^{1,\infty}}; \label{p1} \\
			|I_{\ve,1}^{(2)}| &\lesssim \left\| [p_{\ve}^{(1)}]_{\rm res} \right\|_{L^{\infty}(0,T; L^1(\Omega))} \left( \| \Grad g_{\ve} \|_{L^{\infty}(\Omega)}+ \| 1-g_{\ve}\|_{L^{\infty}(\Omega)} \right) \| \bm{\varphi}\|_{W_0^{1,\infty}} \lesssim  \ve^{m-\alpha} \| \bm{\varphi}\|_{W_0^{1,\infty}};  \label{p2}\\
			|I_{\ve,2}| &\lesssim  \| G\|_{W^{1,\infty}(\Omega)} \left\| \widehat{\vr}_{\ve} \right\|_{L^{\infty}(0,T; L^{\frac{5}{3}}(\Omega))}   \| 1-g_{\ve}\|_{L^{\frac{5}{2}}(\Omega)} \| \bm{\varphi}\|_{W_0^{1,\infty}} \lesssim  \ve^{\frac{6}{5}(\alpha-1)} \| \bm{\varphi}\|_{W_0^{1,\infty}}; \\
			|I_{\ve,3}| &\lesssim   \| G\|_{W^{1,\infty}(\Omega)} \| \bm{\varphi}\|_{W_0^{1,\infty}} |\Omega \setminus \Omega_{\ve}|   \lesssim \ve^{3(\alpha-1)} \| \bm{\varphi}\|_{W_0^{1,\infty}}; \label{B6} \\
			|I_{\ve,4}| &\lesssim  \ve^{m} \| \widehat{\vr}_{\ve}\vuex \|_{L^{\infty} (0,T; L^{\frac{5}{4}}(\Omega))} \| g_{\ve}\|_{L^5(\Omega)} \| \bm{\varphi}\|_{W_0^{1,\infty}} \lesssim  \ve^m  \| \bm{\varphi}\|_{W_0^{1,\infty}}; \label{B1}\\
			|I_{\ve,5}| &\lesssim \ve^{m} \|  \widehat{\vr}_{\ve} \vuex \otimes \vuex \|_{L^1(0,T; L^{\frac{15}{14}}(\Omega))}  \| \Grad g_{\ve} \|_{L^{15}(\Omega)} \| \bm{\varphi}\|_{W_0^{1,\infty}} \lesssim \ve^{m-\frac{1}{5} (4\alpha +1)} \| \bm{\varphi}\|_{W_0^{1,\infty}}; \\
			|I_{\ve,6}| &\lesssim  \ve^{m} \| \vS ( \widehat{\vt}_{\ve}, \Grad \vuex) \|_{L^1(0,T; L^{\frac{3}{2}}(\Omega))}  \| \Grad g_{\ve} \|_{L^3(\Omega)} \| \bm{\varphi}\|_{W_0^{1,\infty}} \lesssim  \ve^{m-1}\| \bm{\varphi}\|_{W_0^{1,\infty}}.
		\end{align}
		Due to hypothesis \eqref{hypothesis parameters}, the exponent for $\ve$ is positive in \eqref{B1}-\eqref{B6}. Therefore, condition \eqref{distribution 2} is satisfied choosing
		\begin{equation} \label{gamma 2}
			\gamma_3:= \min \left\{ m-\alpha, \ \frac{\alpha-3}{2}  \right\}.
		\end{equation}
	\end{proof}
	
	\section{Convergence} \label{Convergence}
	
	From the uniform bounds established in Lemma \ref{uni-bound-Omega}, we deduce the following convergences.
	
	\begin{lemma} \label{Convergences}
		Under the hypotheses of Theorem \ref{main theorem}, the following convergences hold for $\ve \to 0$, passing to suitable subsequences as the case may be.
		\begin{align}
			\widetilde{\vr}_{\ve}^{(1)} \overset{*}{\rightharpoonup} \vr^{(1)} \quad & \mbox{in } L^{\infty}(0,T; L^{\frac{5}{3}}(\Omega)), \label{convergence vr1}\\ 
			\widehat{\vt}_{\ve}^{(1)}\overset{*}{\rightharpoonup}  \vt^{(1)} \quad  & \mbox{in } L^{\infty}(0,T; L^2(\Omega)), \label{convergence vt1 1} \\
			\widehat{\vt}_{\ve}^{(1)}\rightharpoonup \vt^{(1)} \quad & \mbox{in } L^2(0,T; W^{1,2}(\Omega)), \label{convergence vt1 2} \\
			\widehat{\vr}_{\ve} \to \overline{\vr} \quad &\mbox{in } L^{\infty}(0,T; L^{\frac{5}{3}}(\Omega)), \label{c1} \\
			\widehat{\vt}_{\ve} \to \overline{\vt} \quad &\mbox{in } L^{\infty}(0,T; L^2(\Omega)) \cap L^2(0,T; W^{1,2}(\Omega)) \label{c2}, \\
			\widetilde{\vu}_{\ve} \rightharpoonup\vu \quad &\mbox{in } L^2(0,T; W^{1,2}(\Omega; \mathbb{R}^3)), \label{c3} \\
			\widehat{\vr}_{\ve} \widetilde{\vu}_{\ve} \overset{*}{\rightharpoonup} \overline{\vr} \vu \quad &\mbox{in } L^{\infty}(0,T; L^{\frac{5}{4}}(\Omega; \mathbb{R}^3)), \label{c4} \\
			\sqrt{\widehat{\vr}_{\ve}} \widetilde{\vu}_{\ve} \overset{*}{\rightharpoonup} \sqrt{\overline{\vr}} \vu \quad &\mbox{in } L^{\infty}(0,T; L^2(\Omega; \mathbb{R}^3)),\\
			\widehat{\vr}_{\ve} \widetilde{\vu}_{\ve} \otimes \widetilde{\vu}_{\ve} \rightharpoonup \overline{\vr} \ \overline{\vu \otimes \vu} \quad &\mbox{in } L^2(0,T; L^{\frac{30}{29}}(\Omega; \mathbb{R}^{3\times 3})), \label{c5} \\
			\vS(\widehat{\vt}_{\ve}, \Grad \widetilde{\vu}_{\ve}) \rightharpoonup \vS(\overline{\vt}, \nabla_x \vu) \quad &\mbox{in } L^{\frac{5}{4}}(0,T; L^{\frac{5}{4}}(\Omega; \mathbb{R}^{3\times 3})), \label{c6} \\
			\widetilde{p}_{\ve}^{(1)} \rightharpoonup \frac{\partial p(\overline{\vr}, \overline{\vt})}{\partial \vr} \vr^{(1)}+ \frac{\partial p(\overline{\vr}, \overline{\vt})}{\partial \vt} \vt^{(1)} \quad &\mbox{in } L^{\infty}(0,T; L^1(\Omega)). \label{c10} \\
			\widehat{\vr}_{\ve} \widetilde{s}_{\ve}^{(1)} \rightharpoonup \overline{\vr}\left( \frac{\partial s(\overline{\vr}, \overline{\vt})}{\partial \vr} \vr^{(1)}+ \frac{\partial s(\overline{\vr}, \overline{\vt})}{\partial \vt} \vt^{(1)}\right) \quad &\mbox{in } L^2(0,T; L^{\frac{30}{23}}(\Omega), \label{c7} \\
			\widehat{\vr}_{\ve} \widetilde{s}_{\ve}^{(1)} \widetilde{\vu}_{\ve} \rightharpoonup \overline{\vr}\left( \frac{\partial s(\overline{\vr}, \overline{\vt})}{\partial \vr} \vr^{(1)}+ \frac{\partial s(\overline{\vr}, \overline{\vt})}{\partial \vt} \vt^{(1)}\right) \vu  \quad &\mbox{in } L^2(0,T; L^{\frac{30}{29}}(\Omega; \mathbb{R}^3)), \label{c8} 	\\
			\widehat{\kappa}_{\ve} \nabla_x \widehat{\ell}_{\ve}^{(1)} \rightharpoonup \frac{\kappa(\overline{\vt})}{\overline{\vt}} \nabla_x \vt^{(1)} \quad &\mbox{in } L^\frac{14}{13}(0,T; L^{\frac{14}{13}}(\Omega; \mathbb{R}^3)). \label{c9} 
		\end{align}
	\end{lemma}
	
		\begin{remark}
			We point out that we have used the ``bar"-notation for the term $ \overline{\vu \otimes \vu}$ appearing in convergence \eqref{c5} to underline that \textit{a priori} it does not coincide with $\vu \otimes \vu $. Indeed, despite the fact that the limit of the sequence $\{ \widehat{\vr}_{\ve} \widetilde{\vu}_{\ve} \otimes \widetilde{\vu}_{\ve} \}_{\ve >0}$ exists, we cannot conclude that it is equal to $\overline{ \vr} \vu \otimes \vu$, since the convective term is a non-linear function of the density and velocity. More generally, from now on the ``bar"-notation will be used to represent weak limits of non-linear functions of the unknowns.
	\end{remark}
	\begin{proof}
		The main observation used throughout the proof is that the measures of the holes and of the ``residual" subset tend to zero, as it can be deduced from \eqref{e4} and \eqref{c3.1}; specifically,
		\begin{equation} \label{measure tends to zero}
			\esssup_{t\in (0,T)} \left( | \mathcal{M}_{\rm res}|, | \mathcal{M}_{\rm holes}|\right) \to 0 \quad \mbox{as } \ve \to 0.
		\end{equation}
		
		First of all, we have that 
		\begin{equation*}
			\widetilde{\vr}_{\ve}^{(1)}=\vr_{\ve}^{(1)}\mathbbm{1}_{\Omega_{\ve}}= \left[\vr_{\ve}^{(1)}\right]_{\rm ess}+ \left[\vr_{\ve}^{(1)}\right]_{\rm res};
		\end{equation*}
		noticing that, from \eqref{e4}, \eqref{e5}, for a.e. $t\in (0,T)$
		\begin{equation*}
			\left\| \left[ \vr_{\ve}^{(1)}(t) \right]_{\rm res} \right\|_{L^{\frac{5}{3}}(\Omega)}^{\frac{5}{3}}= \left\| \left[ \frac{\vr_{ \ve}(t)-\overline{ \vr}}{\ve^m} \right]_{\rm res}\right\|_{L^{\frac{5}{3}}(\Omega)}^{\frac{5}{3}} \leq \ve^{-\frac{5}{3}m} \left( \| [\vr_{ \ve}(t)]_{\rm res} \|_{L^{\frac{5}{3}}(\Omega)}^{\frac{5}{3}} + \overline{ \vr}^{\frac{5}{3}}|\mathcal{M}_{\rm res}(t)| \right) \leq c(\overline{ \vr})\ve^{\frac{m}{3}},
		\end{equation*}
		using additionally \eqref{e1}, we can deduce, passing to suitable subsequences as the case may be, 
		\begin{align}
			&\left[\vr_{\ve}^{(1)}\right]_{\rm ess} \overset{*}{\rightharpoonup} \vr^{(1)} \quad \mbox{in } L^{\infty}(0,T; L^2(\Omega)), \label{c11} \\
			&\left[\vr_{\ve}^{(1)}\right]_{\rm res} \overset{*}{\rightharpoonup} 0 \hspace{0.75cm} \mbox{in } L^{\infty}(0,T; L^{\frac{5}{3}}(\Omega)),
		\end{align}
		implying, in particular, \eqref{convergence vr1}. At this point, it is straightforward to deduce the strong convergence \eqref{c1}. 
		
		If we now use the decomposition \eqref{holes part}, we can write
		\begin{equation*}
		 \widehat{\vt}_{\ve}^{(1)}=\vt_{\ve}^{(1)}\mathbbm{1}_{\Omega_{\ve}} + E_{\ve} \big( \vt_{ \ve}^{(1)} \big)\mathbbm{1}_{\Omega \setminus \Omega_{\ve}}= \left[\vt_{\ve}^{(1)}\right]_{\rm ess}+ \left[\vt_{\ve}^{(1)}\right]_{\rm res} + \left[\vt_{\ve}^{(1)}\right]_{\rm holes}.
		\end{equation*}
		From \eqref{e4}, \eqref{e5} and the fact that for a.e. $t\in (0,T)$ $\| [\vt_{ \ve}(t)]_{\rm res}\|_{L^2(\Omega)} \leq \| [\vt_{ \ve}(t)]_{\rm res}\|_{L^4(\Omega)} |\mathcal{M}_{\rm res}(t)|^{\frac{1}{4}}$ as consequence of H\"{o}lder's inequality, we have  for a.e. $t\in (0,T)$
		\begin{equation*}
			\left\| \left[ \vt_{\ve}^{(1)}(t) \right]_{\rm res} \right\|_{L^2(\Omega)}^2= \left\| \left[ \frac{\vt_{ \ve}(t)-\overline{ \vt}}{\ve^m} \right]_{\rm res}\right\|_{L^2(\Omega)}^2 \leq \ve^{-2m} \left( \| [\vt_{ \ve}(t)]_{\rm res}\|_{L^4(\Omega)}^2|\mathcal{M}_{\rm res}(t)|^{\frac{1}{2}}  +  \overline{ \vt}^2|\mathcal{M}_{\rm res}(t)| \right) \leq c(\overline{ \vt});
		\end{equation*}
		therefore, using additionally \eqref{e1}, \eqref{measure tends to zero} and estimate \eqref{e11}, we obtain, passing to suitable subsequences as the case may be, 
		\begin{align}
			&\left[\vt_{\ve}^{(1)}\right]_{\rm ess} \overset{*}{\rightharpoonup} \vt^{(1)} \quad \mbox{in } L^{\infty}(0,T; L^2(\Omega)), \label{c2.1}\\
			&\left[\vt_{\ve}^{(1)}\right]_{\rm res} \overset{*}{\rightharpoonup} 0 \hspace{0.75cm} \mbox{in } L^{\infty}(0,T; L^2(\Omega)), \\
			&\left[\vt_{\ve}^{(1)}\right]_{\rm holes} \overset{*}{\rightharpoonup}  0 \hspace{0.5cm} \mbox{in } L^{\infty}(0,T; L^2(\Omega)),
		\end{align}
		implying, in particular, convergence \eqref{convergence vt1 1}; moreover, from \eqref{e3} we recover \eqref{convergence vt1 2}. 
		From \eqref{convergence vt1 1}, \eqref{convergence vt1 2} it is now straightforward to deduce the strong convergence \eqref{c2}. 
		
		Next, convergences \eqref{c3}--\eqref{c5} can be deduced from  \eqref{e9.1}--\eqref{e2} and \eqref{c1}. Similarly, $\vS(\widetilde{\vt}_{\ve}, \Grad \widetilde{\vu}_{\ve})= \vS(\vt_{\ve}, \Grad \vu_{\ve})\mathbbm{1}_{\Omega_{\ve}}$ and therefore, from the constitutive relations \eqref{shear condition}, \eqref{bulk condition} and convergences \eqref{c2}, \eqref{c3} we can deduce \eqref{c6}.
		
		We now point out that for any given function $f\in C^1(\overline{\mathcal{O}}_{\rm ess})$, denoting 
		\begin{equation}
			f_{\ve}^{(1)}: = \frac{f(\vr_{\ve}, \vt_{ \ve})- f(\overline{ \vr}, \overline{ \vt})}{\ve^m},
		\end{equation}
		due to convergences \eqref{c11}, \eqref{c2.1}, we recover that 
		\begin{equation} \label{essential part}
			\left[ f_{\ve}^{(1)} \right]_{\rm ess} \overset{*}{\rightharpoonup} \frac{\partial f(\overline{\vr}, \overline{\vt})}{\partial \vr} \vr^{(1)}+ \frac{\partial f(\overline{\vr}, \overline{\vt})}{\partial \vt} \vt^{(1)} \quad \mbox{in } L^{\infty}(0,T; L^2(\Omega));
		\end{equation}
		see \cite[Proposition 5.2]{FeiNov} for more details. Therefore, writing
		\begin{equation*}
			\widetilde{p}_{\ve}^{(1)} = p_{\ve}^{(1)} \mathbbm{1}_{\Omega_{\ve}} =\left[ p_{\ve}^{(1)} \right]_{\rm ess} + \left[ p_{\ve}^{(1)} \right]_{\rm res},
		\end{equation*}
		where, from \eqref{e4}, \eqref{e15} we have for a.e. $t\in (0,T)$
		\begin{equation*}
			\left\| \left[ p_{\ve}^{(1)}(t) \right]_{\rm res} \right\|_{L^1(\Omega)} \leq \left\| \left[\frac{p(\vr_{ \ve}, \vt_{ \ve})(t)}{\ve^m}\right]_{\rm res}\right\|_{L^1(\Omega)}+\frac{p(\overline{ \vr}, \overline{ \vt})}{\ve^m}|\mathcal{M}_{\rm res}(t)|  \leq c(\overline{\vr},\overline{ \vt}) \ve^m,
		\end{equation*}
		using additionally \eqref{essential part}, we obtain \eqref{c10}. Similarly, we write 
		\begin{equation*}
			\widehat{\vr}_{\ve} \widetilde{s}_{\ve}^{(1)}=\vr_{\ve} s_{\ve}^{(1)} \mathbbm{1}_{\Omega_{\ve}} =\left[ \vr_{\ve} \right]_{\rm ess} \left[  s_{\ve}^{(1)}\right]_{\rm ess} + \left[ \vr_{ \ve} s_{\ve}^{(1)}\right]_{\rm res};
		\end{equation*}
		from \eqref{e4}, \eqref{e5} and \eqref{e6} we get
		\begin{align*}
			&\left\| \left[ \vr_{ \ve} s_{\ve}^{(1)}\right]_{\rm res} \right\|_{L^2(0,T; L^{\frac{30}{23}}(\Omega))}^2 \\ &\lesssim \left\| \left[ \frac{\vr_{ \ve} s(\vr_{ \ve}, \vt_{ \ve})}{\ve^m}  \right]_{\rm res} \right\|_{L^2(0,T; L^{\frac{30}{23}}(\Omega))}^2 + T \frac{s^2(\overline{ \vr}, \overline{ \vt})}{\ve^{2m}} \esssup_{t\in (0,T)}  \big(\| [\vr_{ \ve}(t)]_{\rm res}\|_{L^{\frac{5}{3}}(\Omega)}^2 |\mathcal{M}_{\rm res}(t)|^{\frac{1}{3}}\big) \\
			&\leq c(\overline{ \vr}, \overline{ \vt}) (1+ \ve^{\frac{16}{15}m});
		\end{align*}
		therefore, using additionally \eqref{c1}, \eqref{measure tends to zero} and \eqref{essential part}, we have 
		\begin{align}
			\left[ \vr_{\ve} \right]_{\rm ess} \left[  s_{\ve}^{(1)}\right]_{\rm ess}  & \overset{*}{\rightharpoonup} \overline{\vr}\left( \frac{\partial s(\overline{\vr}, \overline{\vt})}{\partial \vr} \vr^{(1)}+ \frac{\partial s(\overline{\vr}, \overline{\vt})}{\partial \vt} \vt^{(1)}\right) \quad \mbox{in } L^{\infty}(0,T; L^2(\Omega)), \label{c1.1}\\
			\left[ \vr_{ \ve} s_{\ve}^{(1)}\right]_{\rm res} &\rightharpoonup 0 \quad \mbox{in } L^2(0,T; L^{\frac{30}{23}}(\Omega)). \label{c1.3}
		\end{align}
		We get, in particular, \eqref{c7}. In a similar way, we write 
		\begin{equation*}
			\widehat{\vr}_{\ve} \widetilde{s}_{\ve}^{(1)} \widetilde{\vu}_\ve= \left[ \vr_{\ve} \right]_{\rm ess} \left[  s_{\ve}^{(1)}\right]_{\rm ess}  \vuex+ \left[ \vr_{ \ve} s_{\ve}^{(1)}\right]_{\rm res} \vuex;
		\end{equation*}
		from \eqref{e4}, \eqref{e5}, \eqref{e2} and \eqref{e7} we get 
		\begin{align*}
			&\left\| \left[ \vr_{ \ve} s_{\ve}^{(1)}\right]_{\rm res} \vuex \right\|_{L^2(0,T; L^{\frac{30}{29}}(\Omega))}^2 \\ &\lesssim \left\| \left[ \frac{\vr_{ \ve} s(\vr_{ \ve}, \vt_{ \ve})}{\ve^m}  \right]_{\rm res} \vuex \right\|_{L^2(0,T; L^{\frac{30}{29}}(\Omega))}^2 + \frac{s^2(\overline{ \vr}, \overline{ \vt})}{\ve^{2m}}  \esssup_{t\in (0,T)}  \big(\| [\vr_{ \ve}(t)]_{\rm res}\|_{L^{\frac{5}{3}}(\Omega)}^2 |\mathcal{M}_{\rm res}(t)|^{\frac{2}{5}} \big) \| \vuex\|_{L^2(0,T; L^6(\Omega))}^2 \\
			&\leq c(\overline{ \vr}, \overline{ \vt}) (1+ \ve^{\frac{6}{5}m}),
		\end{align*}
		and hence, from \eqref{c3}, \eqref{measure tends to zero}, \eqref{c1.1}, we obtain 
		\begin{align}
			\left[ \vr_{\ve} \right]_{\rm ess} \left[  s_{\ve}^{(1)}\right]_{\rm ess} \vuex & \overset{*}{\rightharpoonup} \overline{\overline{\vr}\left( \frac{\partial s(\overline{\vr}, \overline{\vt})}{\partial \vr} \vr^{(1)}+ \frac{\partial s(\overline{\vr}, \overline{\vt})}{\partial \vt} \vt^{(1)}\right) \vu} \quad \mbox{in } L^2(0,T; L^\frac{3}{2}(\Omega; \mathbb{R}^3)),  \label{limit rsu}\\
			\left[ \vr_{ \ve} s_{\ve}^{(1)}\right]_{\rm res} \vuex &\rightharpoonup 0 \quad \mbox{in } L^2(0,T; L^{\frac{30}{29}}(\Omega; \mathbb{R}^3)).
		\end{align}
		
		Moreover, we can write 
		\begin{align*}
			\widehat{\kappa}_{\ve} \nabla_x  \widehat{\ell}_{\ve}^{(1)} &= \kappa(\vt_{ \ve}) \Grad \ell_{\ve}^{(1)} \mathbbm{1}_{\Omega_{\ve}} + \kappa(\overline{\vt}) \nabla_x E_{\ve} \big( \ell_{\ve}^{(1)}\big) \mathbbm{1}_{\Omega \setminus \Omega_{\ve}} \\[0.1cm]
			&= \left[ \frac{\kappa(\vt_{\ve})}{\vt_{\ve}} \right]_{\rm ess} \Grad \left( \frac{\vt_{\ve}-\overline{\vt}}{\ve^m} \right) + \left[ \frac{\kappa(\vt_{\ve})}{\vt_{\ve}} \Grad \left( \frac{\vt_{\ve}}{\ve^m} \right) \right]_{\rm res}  + \kappa(\overline{\vt}) \Grad [\ell_{\ve}^{(1)}]_{\rm holes } ,
		\end{align*}
		and thus, in virtue of \eqref{e12}, \eqref{e3.1}, \eqref{c2}, \eqref{measure tends to zero}, \eqref{c2.1}, \eqref{convergence vt1 2}  and estimate \eqref{e10}, we get
		\begin{align}
			\left[ \frac{\kappa(\vt_{\ve})}{\vt_{\ve}} \right]_{\rm ess} \Grad \left( \frac{\vt_{\ve}-\overline{\vt}}{\ve^m} \right) &\rightharpoonup \frac{\kappa(\overline{\vt})}{\overline{\vt}} \nabla_x \vt^{(1)} \quad \mbox{in } L^2(0,T; L^2(\Omega; \mathbb{R}^3)), \\
			\left[ \frac{\kappa(\vt_{\ve})}{\vt_{\ve}} \Grad \left( \frac{\vt_{\ve}}{\ve^m} \right) \right]_{\rm res}  &\rightharpoonup 0 \quad \mbox{in } L^\frac{14}{13}(0,T; L^{\frac{14}{13}}(\Omega; \mathbb{R}^3)), \\
			\Grad [\ell_{\ve}^{(1)}]_{\rm holes } &\rightharpoonup 0 \quad \mbox{in } L^2(0,T; L^2(\Omega; \mathbb{R}^3)), \label{c1.6}
		\end{align}
		leading to \eqref{c9}. 
		
		Finally, as consequence of the Div-Curl Lemma \cite[Proposition 3.3]{FeiNov}, we obtain 
		\begin{equation*}
			\overline{\overline{\vr}\left( \frac{\partial s(\overline{\vr}, \overline{\vt})}{\partial \vr} \vr^{(1)}+ \frac{\partial s(\overline{\vr}, \overline{\vt})}{\partial \vt} \vt^{(1)}\right) \vu} = \overline{\vr}\left( \frac{\partial s(\overline{\vr}, \overline{\vt})}{\partial \vr} \vr^{(1)}+ \frac{\partial s(\overline{\vr}, \overline{\vt})}{\partial \vt} \vt^{(1)}\right) \vu
		\end{equation*}
		and hence \eqref{c8}; notice, in particular, that we can repeat the same passages performed in \cite[Section 5.3.2, (iii)]{FeiNov} since only the essential parts of the functions are involved.
	\end{proof}
	
	We are ready to let $\ve \to 0$ in the weak formulations of the problem on the homogenized domain $\Omega$ and get the first result of our work.
	
	\begin{proposition} \label{First result}
		Under the hypotheses of Theorem \ref{main theorem}, passing to suitable subsequences as the case may be,
		\begin{align*}
			\vuex &\rightharpoonup \vu \hspace{0.7cm} \mbox{in } L^2(0,T; W^{1,2}(\Omega; \mathbb{R}^3)), \\
			\widehat{\vt}_{\ve}^{(1)} &\rightharpoonup  \vt^{(1)} \quad \mbox{in } L^2(0,T; W^{1,2}(\Omega)),
		\end{align*}
		where $[\vu, \vt^{(1)}]$ is a dissipative solution to the Oberbeck-Boussinesq system emanating from $[\vu_0, \vt_{0}^{(1)}]$ in the sense of Definition \ref{dissipative solution OB}, with $\vu_0, \vt_{0}^{(1)}$ the weak limits appearing in \eqref{i2}, \eqref{convergence initial temperatures}, respectively.
	\end{proposition}
	
	\begin{proof}
		\textit{Passage to the limit in the continuity equation.} In view of \eqref{i1}, \eqref{c1} and \eqref{c4}, passing to the limit in \eqref{WF1 ext}, we obtain that 
		\begin{equation*}
			\int_{0}^{\tau} \int_{\Omega} \vu \cdot \Grad \varphi \ \dx \dt=0
		\end{equation*}
		holds for any $\tau \in [0,T]$ and any $\varphi \in C^1 ([0,T] \times \overline{\Omega})$; in particular, we get that condition (i) of Definition \ref{dissipative solution OB} is satisfied. Additionally, if we divide \eqref{WF1 ext} by $\ve^m$ and let $\ve \to 0$, from \eqref{i1}, \eqref{c3} and \eqref{convergence vr1} we recover that 
		\begin{equation} \label{C2}
			\left[\int_{\Omega} \vr^{(1)} \varphi(t,\cdot) \ \dx \right]_{t=0}^{t=\tau} =  \int_{0}^{\tau} \int_{\Omega} \left[ \vr^{(1)} \DerTime \varphi + \vr^{(1)} \vu \cdot \Grad \varphi \right] \dx \dt
		\end{equation}
		holds for any $\tau \in [0,T]$ and any $\varphi \in C^1 ([0,T] \times \overline{\Omega})$. Therefore, choosing properly the test function $\varphi$ in \eqref{C2}, from \eqref{integral initial density and temperature} we can deduce that for a.e. $\tau \in (0,T)$
		\begin{equation} \label{integral density zero}
			\int_{\Omega} \vr^{(1)} (\tau, \cdot) \ \dx =0. 
		\end{equation}
		
		\textit{Passage to the limit in the momentum equation.} 
		Putting together convergences \eqref{i1},\eqref{i2}, \eqref{c4}, \eqref{c5}, \eqref{c6} and \eqref{convergence vr1}, we are ready to pass to the limit in \eqref{WF3 ext}, obtaining that the integral identity
			\begin{equation} \label{E1} 
				\begin{aligned}
					\overline{ \vr}\int_{0}^{T} \int_{\Omega} \left(\vu\cdot \DerTime \bfphi +  \overline{\vu \otimes \vu} : \Grad \bfphi \right) \dx \dt = \mu(\overline{ \vt}) &\int_{0}^{T} \int_{\Omega} \left( \Grad \vu + \Grad^{\top} \vu \right) : \Grad \bfphi  \ \dx \dt \\
					- &\int_{0}^{T} \int_{\Omega}  \vr^{(1)} \Grad G \cdot \bfphi \ \dx \dt 
				\end{aligned}
			\end{equation}
			holds for any $\bfphi \in C^1_c((0,T)\times \Omega; \mathbb{R}^3)$ such that $\Div \bfphi =0$. We can now show that 
			\begin{equation*}
				\vu \in C_{\rm weak} ([0,T]; L^2(\Omega; \mathbb{R}^3),
			\end{equation*}
			which, together with \eqref{c3}, provides the regularity class \eqref{regularity class} for $\vu$. 
			To this end, one has to prove that for any fixed $\bm{\phi} \in C^{\infty}_c(\Omega; \mathbb{R}^3)$, the time-dependent function 
			\begin{equation*}
				F_{\bm{\phi}} (t):= \int_{\Omega} \vu(t,x) \cdot \bm{\phi}(x) \ \dx
			\end{equation*} 
			is absolutely continuous on $[0,T]$; the latter will follow by showing that $F_{\bm{\phi}} \in W^{1,q} (0,T)$ for some $q \in [1,\infty]$. Clearly, from the fact that $\vu \in L^{\infty}(0,T; L^2(\Omega; \mathbb{R}^3))$, $F_{\bm{\phi}}$ is bounded. Moreover, introducing the Helmholtz projector $\vv \mapsto \textbf{H}(\vv)$ such that $\Div \textbf{H}(\vv)=0$, it is known that $\textbf{H}$ maps continuously the $L^p$ and $W^{1,p}$-spaces into themselves for any $p\in (1,\infty)$; see e.g. \cite[Section 11.7]{FeiNov}. Using the fact that $\vu$ is solenoidal and therefore $\vu=\textbf{H}(\vu)$, we get that for any $\psi \in C_c^{\infty}(0,T)$ 
			\begin{equation*}
				\int_{0}^{T} F_{\bm{\phi}} (t) \psi'(t) \ \dt = \int_{0}^{T} \int_{\Omega} \textbf{H}(\vu) \cdot \bm{\phi} \ \psi' \ \dx \dt = \int_{0}^{T} \int_{\Omega} \vu \cdot \textbf{H}(\bm{\phi}) \ \psi' \ \dx \dt.
			\end{equation*}
			Finally, using $\bfphi(t,x) = \textbf{H}(\bm{\phi})(x)  \psi(t) $ as test function in \eqref{E1}, form convergences (5.3), (5.6) and (5.15) we get that 
			\begin{equation*}
				\left| \int_{0}^{T} F_{\bm{\phi}} (t) \psi'(t) \ \dt \right| \leq C \| \psi \|_{L^2(0,T)} \quad \mbox{for any} \  \psi \in C_c^{\infty}(0,T),
			\end{equation*}
			implying in particular that $F_{\bm{\phi}} \in W^{1,2}(0,T)$. Next, we introduce the measure 
		\begin{equation} \label{defect R}
			\begin{aligned}
				\mathfrak{R} &\in L^{\infty}(0,T; \mathcal{M}^{+} (\overline{\Omega}; \mathbb{R}^{3\times 3}_{\rm sym})), \\
				\textup{d}\mathfrak{R} &:= \overline{\vr} \left( \overline{\vu \otimes \vu} - \vu \otimes \vu\right) \dx,
			\end{aligned}
		\end{equation}
		where the positivity of $\mathfrak{R}$ follows from the fact that for any $\bm{\xi} \in \mathbb{R}^3$ and any open set $\mathcal{B} \subset \Omega$ we have 
		\begin{align*}
			(\overline{\vu \otimes \vu}- \vu \otimes \vu) : (\bm{\xi} \otimes \bm{\xi}) &= \lim_{\ve \rightarrow 0} \left[(\vuex \otimes \vuex) : (\bm{\xi} \otimes \bm{\xi})\right] - (\vu \otimes \vu) :  (\bm{\xi} \otimes \bm{\xi}) \\
			&= \lim_{\ve\to 0} | \vuex \cdot \bm{\xi}|^2 - |\vu \cdot \bm{\xi}|^2 = \overline{|\vu \cdot \bm{\xi}|^2}- |\vu \cdot \bm{\xi}|^2
		\end{align*}
		in $\mathcal{D}'((0,T) \times \mathcal{B})$ and $\overline{|\vu \cdot \bm{\xi}|^2} \leq |\vu \cdot \bm{\xi}|^2$ due to the convexity of the function $\vu \mapsto |\vu \cdot \bm{\xi}|^2$; see e.g. \cite[Theorem 2.1.1]{Fei}. Noticing that 
		\begin{equation*}
			\int_{\Omega} G \Grad G \cdot \bm{\varphi} \ \dx = \frac{1}{2} \int_{\Omega} \Grad |G|^2 \cdot \bm{\varphi} \ \dx =  - \frac{1}{2} \int_{\Omega} G^2 \  \Div \bm{\varphi} \ \dx =0,
		\end{equation*}
		and, due to \eqref{OB 1} $\Div (\vu \otimes \vu)= \vu \cdot \Grad \vu$, \eqref{E1} can be rewritten as 
		\begin{equation} \label{C3}
			\begin{aligned}
				\overline{\vr}\left[\int_{\Omega} \vu \cdot \bfphi (t,\cdot) \dx \right]_{t=0}^{t=\tau} &= \overline{\vr}\int_{0}^{\tau} \int_{\Omega} \left[ \vu\cdot \DerTime \bfphi - (\vu \cdot \Grad) \vu \cdot  \bfphi  \right] \dx \dt  - \mu(\overline{ \vt}) \int_{0}^{\tau} \int_{\Omega} (\Grad \vu+ \Grad^{\top} \vu ) : \Grad \bm{\varphi} \ \dx \dt \\
				&+ \int_{0}^{\tau} \int_{\Omega} \left(\vr^{(1)} - \frac{\overline{\vr}}{\partial_{\vr}p(\overline{\vr}, \overline{\vt})} G\right)\Grad G \cdot \bfphi \ \dx \dt+ \int_{0}^{\tau} \int_{\overline{\Omega}} \Grad \bm{\varphi}: \textup{d} \mathfrak{R} \ \dt,
			\end{aligned}
		\end{equation}
		for any $\tau \in [0,T]$ and any $\bm{\varphi} \in C^1 ([0,T] \times \overline{\Omega}; \mathbb{R}^3)$, $\bm{\varphi}|_{\partial \Omega}=0$ such that $\Div \bm{\varphi}=0$.
		
		\textit{Passage to the limit in the entropy equation.} Similarly, due to convergences \eqref{i1}, \eqref{convergence initial temperatures}, \eqref{c7}--\eqref{c9}, letting $\ve \to 0$ in \eqref{WF4 ext} we obtain that 
		\begin{equation} \label{C1} 
			\begin{aligned}
				- \int_{\Omega} &\overline{\vr}\left( \frac{\partial s(\overline{\vr}, \overline{\vt})}{\partial \vr} \vr^{(1)}_0+ \frac{\partial s(\overline{\vr}, \overline{\vt})}{\partial \vt} \vt^{(1)}_0\right)  \vphi (0,\cdot) \dx \\
				& = \int_{0}^{T} \int_{\Omega} \left[\overline{\vr}\left( \frac{\partial s(\overline{\vr}, \overline{\vt})}{\partial \vr} \vr^{(1)}+ \frac{\partial s(\overline{\vr}, \overline{\vt})}{\partial \vt} \vt^{(1)}\right)  \big( \DerTime \vphi + \vu \cdot \nabla_x \vphi \big) - \frac{\kappa(\overline{\vt})}{\overline{\vt}} \nabla_x \vt^{(1)} \cdot \Grad \varphi\right] \dx\dt 
			\end{aligned}
		\end{equation}
		holds for any $\vphi \in C_c^1([0,T) \times \overline{\Omega})$. In particular, choosing properly $\varphi$ in \eqref{C1}, from \eqref{integral initial density and temperature} and \eqref{C2} we can deduce that for a.e. $\tau \in (0,T)$
		\begin{equation} \label{integral temperature zero}
			\int_{\Omega} \vt^{(1)} (\tau, \cdot) \ \dx =0. 
		\end{equation}
		
		\textit{Passage to the limit in the Boussinesq equation.} Letting $\ve \to 0$ in \eqref{WF6}, in virtue of convergences \eqref{c1} and \eqref{c10}, we obtain that
		\begin{equation*}
			\int_{0}^{T} \int_{\Omega} \left( \frac{\partial p(\overline{\vr}, \overline{\vt})}{\partial \vr} \vr^{(1)}+ \frac{\partial p(\overline{\vr}, \overline{\vt})}{\partial \vt} \vt^{(1)}\right) \Div \bm{\varphi} \ \dx \dt =  - \overline{\vr} \int_{0}^{T} \int_{\Omega} \Grad G \cdot \bm{\varphi} \ \dx \dt
		\end{equation*}
		holds for any $\bm{\varphi} \in C_c^{\infty}((0,T) \times \Omega; \mathbb{R}^3)$. We get in particular that
		\begin{equation*}
			\Grad \left( \frac{\partial p(\overline{\vr}, \overline{\vt})}{\partial \vr} \vr^{(1)}+ \frac{\partial p(\overline{\vr}, \overline{\vt})}{\partial \vt} \vt^{(1)} \right) = \overline{ \vr} \Grad G \quad \Rightarrow \quad \frac{\partial p(\overline{\vr}, \overline{\vt})}{\partial \vr} \vr^{(1)}+ \frac{\partial p(\overline{\vr}, \overline{\vt})}{\partial \vt} \vt^{(1)} =  \overline{ \vr} G + f(t).
		\end{equation*}
		If we integrate the previous identity over $(0,\tau) \times \Omega$ for any $\tau \in [0,T]$ we can deduce from \eqref{zero mean potential}, \eqref{integral density zero} and \eqref{integral temperature zero} that $f\equiv 0$. Therefore,
		\begin{equation} \label{rho 1}
			\vr^{(1)} = -A \vt^{(1)} + \frac{\overline{ \vr}}{\partial_{\vr}p(\overline{\vr}, \overline{\vt}) }G,
		\end{equation}
		where $A$ is the constant defined in \eqref{OB 4}.
		
		We can now substitute \eqref{rho 1} into \eqref{C3} and \eqref{C1}; we obtain that condition (ii) of Definition \ref{dissipative solution OB} is satisfied and
		\begin{equation*}
			-\overline{ \vr}c_p \int_{\Omega}  \vt_0^{(1)} \ \varphi(0,\cdot) \ \dx= \int_{0}^{T} \int_{\Omega} \left[ \overline{ \vr} c_p \vt^{(1)} (\DerTime \varphi + \vu \cdot \Grad \varphi) - (\kappa(\overline{ \vt}) \Grad \vt^{(1)} +\overline{ \vt} A G\vu) \cdot \Grad \varphi\right] \dx \dt
		\end{equation*}
		holds for any $\vphi \in C_c^1([0,T) \times \overline{\Omega})$. In particular, we have used the fact that, from Gibb's relation \eqref{Gibb's relation},
		\begin{equation*}
			\frac{\partial s(\overline{\vr}, \overline{\vt})}{\partial \vr} 
			=- \frac{1}{\overline{ \vr}^2} \frac{\partial p(\overline{\vr}, \overline{\vt})}{\partial \vt},
		\end{equation*}
		and, since the initial data $(\vr_0^{(1)}, \vt_{0}^{(1)})$ are well-prepared and satisfy \eqref{well-prepared data}, 
		\begin{equation*}
			c_p \vt_0^{(1)} = \overline{ \vt} \left( \frac{\partial s(\overline{\vr}, \overline{\vt})}{\partial \vr} \vr^{(1)}_0+ \frac{\partial s(\overline{\vr}, \overline{\vt})}{\partial \vt} \vt^{(1)}_0 + a(\overline{\vr}, \overline{\vt}) G \right).
		\end{equation*}
		Hence, $\vt^{(1)}$ satisfies the weak formulation of \eqref{OB 2}.
		
		Next, by interpolation, from \eqref{c3} and \eqref{c4} we can deduce that $\vu\in L^{\frac{10}{3}}(0,T; L^{\frac{10}{3}}(\Omega; \mathbb{R}^3))$, implying that 
		\begin{equation*}
			\vu \cdot \Grad \vt^{(1)} \in L^p(0,T; L^p(\Omega)) \quad \mbox{with} \quad p=\frac{5}{4}.
		\end{equation*}
		Due to the additional assumption \eqref{Slobodeckii condition} on the initial temperature $\vt_0^{(1)}$, we can apply \cite[Theorem 10.22]{FeiNov} to the deduce the regularity class \eqref{regularity class} for $\vt^{(1)}$. Consequently, condition (iii) of Definition \ref{dissipative solution OB} is satisfied.
		
		\textit{Passage to the limit in the energy equality.}  
		We start pursuing the same idea developed in \cite[Section 3.2]{PokSkr}. From the constitutive relation \eqref{conductivity condition}, we can write for any $\delta \in (0,1)$ and any $\psi \in C^1[0,T]$, $\psi\geq 0$
			\begin{equation} \label{e16}
				\begin{aligned}
					\int_{0}^{T} &\int_{\Omega_{\ve}} \kappa(\vt_{ \ve}) | \Grad \ell_{\ve}^{(1)} |^2 \psi \  \dx \dt \\
					=& \ \int_{0}^{T} \int_{\Omega} \widehat{\kappa}_{\ve} | \Grad \widehat{\ell}_{\ve}^{(1)} |^{2-\delta} \psi \ \dx \dt - \kappa(\overline{ \vt}) \int_{0}^{T} \int_{\Omega \setminus \Omega_{\ve}} | \Grad E_{\ve} (\ell_{\ve}^{(1)}) |^{2-\delta} \psi \ \dx \dt\\
					&+ \int_{0}^{T} \int_{\Omega_{\ve}} \kappa(\vt_{ \ve}) \left(| \Grad \ell_{\ve}^{(1)} |^2 - | \Grad \ell_{\ve}^{(1)} |^{2-\delta}\right) \psi \ \dx \dt  \\
					\geq & \  \int_{0}^{T}  \int_{\Omega} \left[ \frac{\kappa(\vt_{\ve})}{\vt_{\ve}^{2-\delta}} \right]_{\rm ess} \big| \Grad \widehat{\vt}_{\ve}^{(1)} \big|^{2-\delta} \psi \ \dx \dt - \kappa(\overline{ \vt}) \int_{0}^{T} \int_{\Omega \setminus \Omega_{\ve}} | \Grad E_{\ve} (\ell_{\ve}^{(1)}) |^{2-\delta} \psi\  \dx \dt\\ 
					& - \overline{\kappa}\int_{0}^{T} \int_{\Omega_{\ve}} (1+\vt_{ \ve}^3) \left( | \Grad \ell_{\ve}^{(1)} |^{2-\delta} -| \Grad \ell_{\ve}^{(1)} |^2 \right) \mathbbm{1}_{\{  |\Grad \ell_{\ve}^{(1)} | \leq 1 \}} \psi \  \dx \dt = \sum_{k=1}^{3} I_{\ve, k}. 
				\end{aligned}
			\end{equation}
			From the properties of the extension operator $E_{\ve} $, we have that 
			\begin{equation}
				|I_{\ve,2}| \leq c(\psi) \| \widehat{\ell}_{\ve}^{(1)} \|_{L^2(0,T; W^{1,2}(\Omega; \mathbb{R}^3 ))} |\Omega \setminus \Omega_{\ve}|^{\frac{\delta}{2}} \lesssim \ve^{\frac{3}{2}\delta(\alpha-1)};
			\end{equation}
			moreover, noticing that from \eqref{c2} by interpolation
			\begin{equation*}
				\{\widehat{\vt}_{\ve}\}_{\ve>0} \mbox{ is uniformly bounded } \in L^p(0,T;L^p(\Omega)) \ \mbox{with } p= \frac{10}{3}
			\end{equation*}
			and that the function $f(z)= (z^{2-\delta} - z^2)^{10}$ attains its maximum on $[0,1]$ at the point $z_{\max} = \left( 1- \frac{\delta}{2}\right)^{\frac{1}{\delta}}$, we deduce that
			\begin{equation}
				|I_{\ve,3}|  \leq c(\psi) \left(1+ \| \widehat{\vt}_{ \ve}\|_{L^p(0,T;L^p(\Omega ))}\right)  I(\delta) \lesssim I(\delta),
			\end{equation}
			with 
			\begin{equation}
				I(\delta):= \frac{\delta}{2} \left(1- \frac{\delta}{2}\right)^{\frac{2}{\delta}-1}.
			\end{equation}
			On the other hand,
			\begin{equation*}
				\vS (\vt_{\ve}, \Grad \vu_{\ve}): \Grad \vu_{\ve} = \frac{\mu(\vt_{\ve})}{2} \left| \Grad \vu_{\ve} + \Grad^{\top} \vu_{\ve} -\frac{2}{3} (\Div \vu_{\ve}) \mathbb{I}\right|^2 + \eta(\vt_{\ve}) |\Div \vu_{\ve}|^2,
			\end{equation*}
			and therefore, for any $\psi \in C^1[0,T]$, $\psi\geq 0$
			\begin{equation} \label{e17}
				\int_{0}^{T} \int_{\Omega} \widehat{\vt}_{ \ve}^{-1} \vS (\widehat{\vt}_{\ve}, \Grad \widetilde{\vu}_{\ve}): \Grad \widetilde{\vu}_{\ve} \ \psi  \ \dx \dt \geq \frac{1}{2} \int_{0}^{T} \int_{\Omega}  \left[ \frac{\mu(\vt_{\ve})}{\vt_{\ve}} \right]_{\rm ess} \left| \Grad \widetilde{\vu}_{\ve} + \Grad^{\top} \widetilde{\vu}_{\ve} -\frac{2}{3} (\Div \widetilde{\vu}_{\ve}) \mathbb{I}\right|^2 \psi \  \dx \dt.
			\end{equation}
			Putting together \eqref{e16}--\eqref{e17}, from \eqref{c2}, \eqref{c3}, \eqref{convergence vt1 2} and the lower semi-continuity of convex functions, for any $\psi \in C^1[0,T]$, $\psi\geq 0$ we have
			\begin{align*}
				&\liminf_{\ve\to 0} \int_{0}^{T} \int_{\Omega_{\ve}} \overline{\vt}\left[ \frac{1}{\vt_{\ve}} \vS(\vt_{\ve}, \Grad \vu_{\ve}): \nabla_x \vu_{\ve} + \kappa(\vt_{\ve}) \left| \nabla_x \ell^{(1)}_{\ve}\right|^2 \right] \psi  \ \dx\dt \\
				&\geq \frac{\mu(\overline{ \vt})}{2} \int_{0}^{T} \int_{\Omega} |\Grad \vu+ \Grad^{\top} \vu |^2 \ \psi \ \dx \dt + 
				\overline{ \vt}^{\delta-1}\kappa(\overline{ \vt}) \int_{0}^{T} \int_{\Omega} |\Grad \vt^{(1)}|^{2-\delta} \ \psi \ \dx \dt- I(\delta).
			\end{align*} 
			Similarly to \cite[Section 3.2]{PokSkr}, if now we let $\delta \to 0$, $I(\delta) \to 0$ and, by Vitali convergence Theorem, we can conclude that 
			\begin{equation*}
				\overline{ \vt}^{\delta-1}\kappa(\overline{ \vt}) \int_{0}^{T} \int_{\Omega} |\Grad \vt^{(1)}|^{2-\delta} \ \psi \ \dx \dt \to \frac{\kappa(\overline{ \vt})}{\overline{ \vt}}\int_{0}^{T} \int_{\Omega} |\Grad \vt^{(1)}|^{2} \ \psi \ \dx \dt.
			\end{equation*}
		Introducing the measure 
		\begin{equation} \label{defect E}
			\begin{aligned}
				\mathfrak{E} &\in L^{\infty}(0,T; \mathcal{M}^+ (\overline{\Omega})), \\
				\textup{d} \mathfrak{E} &:= \frac{1}{2} \overline{ \vr} \left( \overline{|\vU|^2} - |\vU|^2 \right) \dx,
			\end{aligned}
		\end{equation}
		and noticing that 
		\begin{equation*}
			\widetilde{H}_{\overline{\vt}, \ve}- (\widetilde{\vr}_{\ve}-\overline{\vr}) \frac{\partial H_{\overline{\vt}}(\overline{\vr}, \overline{\vt})}{\partial \vr} -H_{\overline{\vt}}(\overline{\vr}, \overline{\vt}) = \begin{cases}
				H_{\overline{ \vt}}(\vr_{ \ve}, \vt_{ \ve})- (\vr_{ \ve}-\overline{\vr}) \frac{\partial H_{\overline{\vt}}(\overline{\vr}, \overline{\vt})}{\partial \vr} -H_{\overline{\vt}}(\overline{\vr}, \overline{\vt}) &\mbox{in } \Omega_{\ve}, \\
				0 &\mbox{in } \Omega \setminus \Omega_{\ve},
			\end{cases}
		\end{equation*}
		we can pass to the limit in \eqref{energy inequality with Helmholtz} and repeat the same passages as in \cite[Section 5.5.4]{FeiNov}, obtaining that the energy inequality 
		\begin{align*}
			&\int_{\Omega} \left[ \frac{1}{2} \overline{ \vr}|\vu|^2+ \frac{1}{2\overline{ \vr}} \frac{\partial p(\overline{ \vr}, \overline{ \vt})}{\partial \vr} \big| \vr^{(1)} \big|^2+ \frac{\overline{ \vr}}{2\overline{ \vt}} \frac{\partial e(\overline{ \vr}, \overline{ \vt})}{\partial \vt} \big| \vt^{(1)} \big|^2 - \vr^{(1)}  G \right] (\tau, \cdot) \dx \\
			+ \frac{\mu(\overline{ \vt})}{2}& \int_{0}^{\tau} \int_{\Omega} |\Grad \vu+ \Grad^{\top} \vu |^2 \ \dx \dt + \frac{\kappa(\overline{ \vt})}{\overline{ \vt}} \int_{0}^{\tau} \int_{\Omega} |\Grad \vt^{(1)} |^2 \ \dx \dt + \int_{\Omega} \textup{d} \mathfrak{E}(\tau) \\
			\leq &\int_{\Omega} \left[ \frac{1}{2} \overline{ \vr}|\vu_{0}|^2+ \frac{1}{2\overline{ \vr}} \frac{\partial p(\overline{ \vr}, \overline{ \vt})}{\partial \vr} \big| \vr_0^{(1)} \big|^2+ \frac{\overline{ \vr}}{2\overline{ \vt}} \frac{\partial e(\overline{ \vr}, \overline{ \vt})}{\partial \vt} \big| \vt_0^{(1)} \big|^2 - \vr_0^{(1)}  G \right] \dx
		\end{align*}
		holds for a.e. $\tau \in (0,T)$. Substituting \eqref{well-prepared data} and \eqref{rho 1}, we get that condition (iv) of Definition \ref{dissipative solution OB} is satisfied. 
		
		Finally, condition (v) of Definition \ref{dissipative solution OB} follows from the fact that $\trace [\vu \otimes \vu] = |\vu|^2 $; this concludes the proof.
	\end{proof}
	
	\section{Weak--strong uniqueness} \label{Weak--strong uniqueness principle}
	
	Our goal in this section is to prove the \textit{weak--strong uniqueness} principle for the target system: if the Oberbeck--Boussinesq approximation admits a strong solution, then it must coincide with the dissipative solution emanating from the same initial data.
	
	We start recalling the following result on the local existence of strong solutions, cf. \cite[Theorem 2.1]{CliGuiRoj}. Notice that in \cite{CliGuiRoj} the authors considered time-periodic solutions with small data; however, the proof, based on Galerkin approximation and uniform bounds, can be adapted to get local existence with large data. We also recall the recent result by Abbatiello and Feireisl \cite{AbbFei} where the existence was proven considering non-local boundary conditions for the temperature.
	
	\begin{theorem}[Existence of strong solutions to the Oberbeck--Boussinesq system] \label{Existence strong solutions OB}
		There exists a positive time $T^{*}$ and a trio of functions
		\begin{align}
			\vU &\in W^{1,2}(0,T^{*}; L^2(\Omega; \mathbb{R}^3)) \cap L^{\infty}(0,T^{*}; W^{1,2}(\Omega; \mathbb{R}^3)) \cap L^2(0,T^{*}; W^{2,2} (\Omega; \mathbb{R}^3)), \label{reg class 1}\\
			\Theta &\in W^{1,2}(0,T^{*}; W^{1,2}(\Omega)) \cap L^{\infty}(0,T^{*}; W^{2,2}(\Omega)) \cap L^2(0,T^{*}; W^{2,3} (\Omega)), \\
			\Pi &\in L^2(0,T; W^{1,2}(\Omega)) \label{reg class 3}
		\end{align}
		satisfying the Oberbeck--Boussinesq system \eqref{OB 1}--\eqref{OB 7} a.e. in $(0,T^{*})\times \Omega$.
	\end{theorem}
	
	\begin{theorem} [Weak--strong uniqueness principle] \label{Weak-strong uniqueness}
		Let $[\vU, \Theta, \Pi]$ be a strong solution of the Oberbeck-Boussinesq system \eqref{OB 1}--\eqref{OB 7} on $[0,T^*]$, the existence of which is guaranteed by Theorem \ref{Existence strong solutions OB}. Let $[\vu, \vartheta^{(1)}]$ be a dissipative solution of the same system with dissipation defects $\mathfrak{R}, \mathfrak{E}$ in the sense of Definition \ref{dissipative solution OB}. If 
		\begin{equation} \label{same initial data}
			[\vU(0,x), \Theta(0,x)] = [\vu(0,x), \vt^{(1)}(0,x)] \quad \mbox{for a.e. } x\in \Omega 
		\end{equation}
		then $\mathfrak{R}\equiv \mathfrak{E} \equiv 0$ and 
		\begin{equation} \label{a.e. equality}
			[\vU(t,x), \Theta(t,x)] = [\vu(t,x), \vt^{(1)}(t,x)] \quad \mbox{for a.e. } (t,x)\in (0,T^*) \times \Omega.
		\end{equation}
	\end{theorem}
	\begin{proof}
		Let us define 
		\begin{equation*}
			E(\vu, \vt^{(1)} \ | \ \vU, \Theta):= \frac{1}{2} \left( \overline{ \vr}|\vu -\vU|^2 + \frac{\overline{ \vr}}{\overline{ \vt}} c_p | \vt^{(1)}- \Theta|^2\right)
		\end{equation*}
		and for any $\tau \in [0,T^*]$ the spatial integral of it, known as \textit{relative energy functional}, 
		\begin{equation*}
			\begin{aligned}
				\mathcal{E}(\vu, \vt^{(1)} \ | \ \vU, \Theta)(\tau)&:= \int_{\Omega} E(\vu, \vt^{(1)} \ | \ \vU, \Theta)(\tau, \cdot) \ \dx \\
				&= \frac{1}{2} \int_{\Omega} \left( \overline{ \vr}|\vu -\vU|^2 + \frac{\overline{ \vr}}{\overline{ \vt}} c_p | \vt^{(1)}- \Theta|^2\right) (\tau, \cdot ) \ \dx.
			\end{aligned}
		\end{equation*}
		Clearly, $\mathcal{E}(\vu, \vt^{(1)} \ | \ \vU, \Theta)(\tau) \geq 0$  for any $\tau \in [0,T^*]$ and the equality holds if and only if \eqref{a.e. equality} holds. Therefore, it is enough to show that 
		\begin{equation} \label{zero energy}
			\mathfrak{E} \equiv 0, \quad \mathcal{E}(\vu, \vt^{(1)} \ | \ \vU, \Theta) \equiv 0 \quad \mbox{a.e. in } (0,T^*). 
		\end{equation}
		
		Let us at first suppose that $[\vU, \Theta, \Pi]$ are smooth and compactly supported functions such that $\vU|_{\partial \Omega}=0$ and $\Div \vU=0$. Then, $\bm{\varphi}= \vU$ can be used as test function in the weak formulation \eqref{DF OB 1}, obtaining 
		\begin{equation} \label{Id 1}
			\begin{aligned}
				\overline{\vr}\left[\int_{\Omega} (\vu \cdot \vU) (t, \cdot) \ \dx\right]_{t=0}^{t=\tau} &= \overline{\vr} \int_{0}^{\tau} \int_{\Omega}   \left( \vu \cdot \DerTime \vU + (\vu\otimes \vu): \Grad \vU \right) \dx \dt \\
				&- 2\mu(\overline{ \vt}) \int_{0}^{\tau} \int_{\Omega} \mathbb{D}_x \vu : \mathbb{D}_x \vU \ \dx \dt - A\int_{0}^{\tau} \int_{\Omega} \vt^{(1)}\Grad G \cdot \vU \ \dx \dt \\
				&+ \int_{0}^{\tau} \int_{\overline{\Omega}} \Grad \vU: \textup{d} \mathfrak{R} \ \dt, 
			\end{aligned}
		\end{equation}
		where we have introduced symmetric velocity gradient, defined as
		\begin{equation*}
			\mathbb{D}_x \vv = \frac{\Grad \vv + \Grad^{\top} \vv}{2}.
		\end{equation*}
		Similarly, $\varphi = \Theta$ can be used as test function in the weak formulation of \eqref{OB 2}, obtaining 
		\begin{equation} \label{Id 2}
			\begin{aligned}
				\frac{\overline{\vr}}{\overline{ \vt}} c_p\left[\int_{\Omega}( \vt^{(1)}  \Theta )(t, \cdot) \ \dx\right]_{t=0}^{t=\tau} &= \frac{\overline{\vr}}{\overline{ \vt}} c_p \int_{0}^{\tau} \int_{\Omega} \vt^{(1)} \left( \DerTime \Theta +  \vu \cdot  \Grad \Theta \right) \dx \dt \\
				&- \frac{\kappa(\overline{ \vt})}{\overline{ \vt}} \int_{0}^{\tau} \int_{\Omega} \Grad \vt^{(1)} \cdot \Grad \Theta \ \dx \dt + A\int_{0}^{\tau} \int_{\Omega} \Theta\Grad G \cdot \vu \ \dx \dt . 
			\end{aligned}
		\end{equation}
		Moreover, using $\varphi = |\vU|^2, |\Theta|^2$ as test function in the weak formulation of the incompressibility condition \eqref{OB 1}, we have the following identity
		\begin{equation} \label{Id 3}
			\begin{aligned}
				\frac{1}{2} \left[ \int_{\Omega} \left( \overline{ \vr}|\vU|^2 + \frac{\overline{ \vr}}{\overline{ \vt}} c_p | \Theta|^2\right)(t,\cdot) \right]_{t=0}^{t=\tau} &= \overline{ \vr} \int_{0}^{\tau} \int_{\Omega}\left[\vU \cdot \DerTime \vU + (\vu \cdot \Grad) \vU \cdot \vU \right] \ \dx \dt \\
				&+ \frac{\overline{ \vr}}{\overline{ \vt}} c_p \int_{0}^{\tau} \int_{\Omega} \Theta (\DerTime \Theta + \vu \cdot \Grad \Theta)\ \dx \dt.
			\end{aligned}
		\end{equation}
		We can now subtract \eqref{Id 1}, \eqref{Id 2} and sum \eqref{Id 3} to the energy inequality \eqref{DF OB 2}, obtaining
		\begin{equation*} 
			\begin{aligned}
				&\left[ \mathcal{E}(\vu, \vt^{(1)} \ | \ \vU, \Theta) (t) \right]_{t=0}^{t=\tau} + \int_{\overline{\Omega}} \textup{d} \mathfrak{E}(\tau) \\
				&\quad + 2\mu(\overline{ \vt}) \int_{0}^{\tau} \int_{\Omega} \mathbb{D}_x \vu : \mathbb{D}_x (\vu -\vU) \ \dx \dt + \frac{\kappa(\overline{ \vt})}{\overline{ \vt}} \int_{0}^{\tau} \int_{\Omega} \Grad \vt^{(1)} \cdot \Grad (\vt^{(1)} - \Theta) \ \dx \dt \\
				& \leq - \int_{0}^{\tau} \int_{\Omega} (\vu-\vU) \cdot \left( \overline{ \vr} [\DerTime \vU + (\vU \cdot \Grad ) \vU]+ A \Theta \Grad G\right) \dx \dt \\
				& \quad - \frac{1}{\overline{ \vt}} \int_{0}^{\tau} \int_{\Omega} (\vt^{(1)}- \Theta) \left[ \overline{ \vr} c_p (\DerTime \Theta +  \vU \cdot \Grad \Theta) - \overline{ \vt} A \Grad G \cdot \vU \right] \dx \dt \\
				&\quad - \overline{ \vr} \int_{0}^{\tau} \int_{\Omega} \left[(\vu-\vU) \cdot \Grad  \right] \vU \cdot (\vu -\vU) \ \dx \dt - \frac{\overline{\vr}}{\overline{ \vt}} c_p \int_{0}^{\tau} \int_{\Omega}  (\vt^{(1)}- \Theta)  \Grad \Theta \cdot (\vu -\vU) \ \dx \dt \\
				& \quad - \int_{0}^{\tau} \int_{\overline{\Omega}} \Grad \vU: \textup{d} \mathfrak{R} \ \dt.
			\end{aligned}
		\end{equation*}
		Next, we add to the previous inequality the following vanishing integrals,
		\begin{align*}
			\mu(\overline{ \vt}) &\int_{0}^{\tau} \int_{\Omega} \left[(\vu-\vU) \cdot \Delta_x \vU  + 2\mathbb{D}_x(\vu-\vU) : \mathbb{D}_x \vU\right]  \dx \dt, \\
			\frac{\kappa(\overline{ \vt})}{\overline{ \vt}} &\int_{0}^{\tau} \int_{\Omega} \left[(\vt^{(1)}-\Theta) \Delta_x \Theta  + \Grad(\vt^{(1)}-\Theta) \cdot \Grad \Theta \right]  \dx \dt, \\
			&\int_{0}^{\tau} \int_{\Omega} (\vu-\vU) \cdot \Grad \Pi  \ \dx \dt,
		\end{align*}
		getting finally the \textit{relative energy inequality}, 
		\begin{equation} \label{relative energy inequality}
			\begin{aligned}
				&\left[ \mathcal{E}(\vu, \vt^{(1)} \ | \ \vU, \Theta) (t) \right]_{t=0}^{t=\tau} + \int_{\overline{\Omega}} \textup{d} \mathfrak{E}(\tau) \\
				&\quad +2\mu(\overline{ \vt}) \int_{0}^{\tau} \int_{\Omega} | \mathbb{D}_x (\vu - \vU)|^2 \ \dx \dt + \frac{\kappa(\overline{ \vt})}{\overline{ \vt}} \int_{0}^{\tau} \int_{\Omega} |\Grad (\vt^{(1)} - \Theta)|^2 \ \dx \dt \\
				& \leq - \int_{0}^{\tau} \int_{\Omega} (\vu-\vU) \cdot \left( \overline{ \vr} [\DerTime \vU + (\vU \cdot \Grad) \vU] + \Grad \Pi - \mu(\overline{ \vt}) \Delta_x \vU + A \Theta \Grad G\right) \dx \dt \\
				& \quad - \frac{1}{\overline{ \vt}} \int_{0}^{\tau} \int_{\Omega} (\vt^{(1)}- \Theta) \left[ \overline{ \vr} c_p (\DerTime \Theta + \vU \cdot \Grad \Theta ) - \kappa(\overline{ \vt}) \Delta_x \Theta - \overline{ \vt} A \Grad G \cdot \vU \right] \dx \dt \\
				&\quad - \overline{ \vr} \int_{0}^{\tau} \int_{\Omega}  \left[(\vu-\vU) \cdot \Grad  \right] \vU \cdot (\vu -\vU) \ \dx \dt - \frac{\overline{\vr}}{\overline{ \vt}} c_p \int_{0}^{\tau} \int_{\Omega}  (\vt^{(1)}- \Theta)  \Grad \Theta \cdot (\vu -\vU) \ \dx \dt \\
				& \quad - \int_{0}^{\tau} \int_{\overline{\Omega}} \Grad \vU: \textup{d} \mathfrak{R} \ \dt.
			\end{aligned}
		\end{equation}
		
		The class of functions $[\vU, \Theta, \Pi]$ satisfying the relative energy inequality can be enlarged by a density argument, as long as all the involved integrals remain well-defined. In particular \eqref{relative energy inequality} holds for $[\vU, \Theta, \Pi]$ belonging to the regularity classes defined in \eqref{reg class 1}--\eqref{reg class 3}.
		
		If we additionally suppose that $[\vU, \Theta, \Pi]$ is a strong solution of \eqref{OB 1}--\eqref{OB 7} satisfying \eqref{same initial data}, we get that $\mathcal{E}(\vu, \vt^{(1)} \ | \ \vU, \Theta) (0)$ and the first two integrals on the right-hand side of \eqref{relative energy inequality} vanish; in particular, \eqref{relative energy inequality} reduces to
		\begin{equation} \label{REI}
			\begin{aligned}
				& \mathcal{E}(\vu, \vt^{(1)} \ | \ \vU, \Theta) (\tau)  + \int_{\overline{\Omega}} \textup{d} \mathfrak{E}(\tau) \\
				&\quad + 2\mu(\overline{ \vt}) \int_{0}^{\tau} \int_{\Omega} | \mathbb{D}_x (\vu - \vU)|^2 \ \dx \dt + \frac{\kappa(\overline{ \vt})}{\overline{ \vt}} \int_{0}^{\tau} \int_{\Omega} |\Grad (\vt^{(1)} - \Theta)|^2 \ \dx \dt \\
				& \leq - \overline{ \vr} \int_{0}^{\tau} \int_{\Omega}  [(\vu-\vU) \otimes (\vu -\vU)] : \Grad \vU \ \dx \dt - \frac{\overline{\vr}}{\overline{ \vt}} c_p \int_{0}^{\tau} \int_{\Omega} (\vt^{(1)}- \Theta)  \Grad \Theta \cdot (\vu -\vU)  \ \dx \dt \\
				& \quad - \int_{0}^{\tau} \int_{\overline{\Omega}} \Grad \vU: \textup{d} \mathfrak{R} \ \dt,
			\end{aligned}
		\end{equation}
		for a.e $\tau \in (0,T)$. Clearly, 
		\begin{equation*}
			\overline{ \vr} | (\vu-\vU) \otimes (\vu -\vU) | \lesssim \frac{1}{2} \overline{ \vr} \trace [(\vu-\vU) \otimes (\vu -\vU)] = \frac{1}{2} \overline{ \vr}  |\vu -\vU|^2,
		\end{equation*}
		\begin{equation*}
			\frac{\overline{\vr}}{\overline{ \vt}} c_p| (\vu -\vU) (\vt^{(1)}- \Theta) | \lesssim \frac{1}{2} \overline{ \vr}  |\vu -\vU|^2 + \frac{1}{2} 	\frac{\overline{\vr}}{\overline{ \vt}} c_p |\vt^{(1)}- \Theta|^2,
		\end{equation*}
		\begin{equation*}
			|\mathfrak{R}| \lesssim \trace [\mathfrak{R}] \lesssim \mathfrak{E},
		\end{equation*}
		where in the last inequality we have used the compatibility condition \eqref{compatibility condition}. Therefore, for a.e. $\tau \in (0,T)$ we obtain 
		\begin{equation*}
			\mathcal{E}(\vu, \vt^{(1)} \ | \ \vU, \Theta) (\tau)  + \int_{\overline{\Omega}} \textup{d} \mathfrak{E}(\tau) \leq c(\Grad \vU, \Grad \Theta) \int_{0}^{\tau} \left( \mathcal{E}(\vu, \vt^{(1)} \ | \ \vU, \Theta)(t)  + \int_{\overline{\Omega}} \textup{d} \mathfrak{E}(t)  \right) \dt.
		\end{equation*}
		Applying the Gronwall argument, we recover in particular that for a.e. $\tau \in (0,T)$
		\begin{equation*}
			\mathcal{E}(\vu, \vt^{(1)} \ | \ \vU, \Theta) (\tau)  + \int_{\overline{\Omega}} \textup{d} \mathfrak{E}(\tau) \leq 0.
		\end{equation*}
		Since the left-hand side of the previous inequality is the sum of two non-negative quantities, the only possibility is that \eqref{zero energy} holds; we get the claim.
	\end{proof}
	
	\subsection{Proof of Theorem \ref{main theorem}}
	
	In Proposition \ref{First result}, we have proven that, passing to suitable subsequences as the case may be,
	\begin{equation}
		[\vuex, \vtex^{(1)}] \rightharpoonup [\vu, \vt^{(1)}] \quad \mbox{in } L^2(0,T; W^{1,2}(\Omega; \mathbb{R}^4)),
	\end{equation}
	where $[\vu, \vt^{(1)}]$ is a dissipative solution to the Oberbeck-Boussinesq system in the sense of Definition \ref{dissipative solution OB} with dissipation defects $\mathfrak{R}$, $\mathfrak{E}$ defined by \eqref{defect R}, \eqref{defect E}, respectively. From the fact that $[\vU_0, \Theta_0]= [\vu_0, \theta_0^{(1)}]$, Theorem \ref{main theorem} is therefore a straightforward corollary of Theorem \ref{Weak-strong uniqueness}.

\end{document}